\documentclass[10pt]{amsart}
      \usepackage[mathscr]{eucal}
      \usepackage{amsmath,amsfonts}
      \parskip\smallskipamount
          \newtheorem{theorem}{Theorem}[section]
      
      \newtheorem{proposition}[theorem]{Proposition}
      \newtheorem{corollary}[theorem]{Corollary}
      \newtheorem{lemma}[theorem]{Lemma}

      \makeatletter
      \@addtoreset{equation}{section}
      \makeatother

      \newcommand{\BB}{{\mathbb B}}
      \newcommand{\CC}{{\mathbb C}}
      \newcommand{\NN}{{\mathbb N}}

      \newcommand{\DD}{{\mathbb D}}
      \newcommand{\RR}{{\mathbb R}}
      \newcommand{\FF}{{\mathbb F}}
      \newcommand{\TT}{{\mathbb T}}

      \newcommand{\cA}{{\mathcal A}}

      \newcommand{\cD}{{\mathcal D}}
      \newcommand{\cE}{{\mathcal E}}
      
      \newcommand{\cG}{{\mathcal G}}
      \newcommand{\cH}{{\mathcal H}}
      \newcommand{\cK}{{\mathcal K}}
      \newcommand{\cL}{{\mathcal L}}
      \newcommand{\cM}{{\mathcal M}}
      
      \newcommand{\cP}{{\mathcal P}}
      \newcommand{\cR}{{\mathcal R}}

      \newdimen\expt
      \def\boxit#1{\setbox0\hbox{$\displaystyle{#1}$}
            \hbox{\lower.4\expt
       \hbox{\lower3\expt\hbox{\lower\dp0
            \hbox{\vbox{\hrule height.4\expt
       \hbox{\vrule width.4\expt\hskip3\expt
            \vbox{\vskip3\expt\box0\vskip2\expt}%
       \hskip3\expt\vrule width.4\expt}\hrule height.4\expt}}}}}}
      
\begin{document}
       \pagestyle{myheadings}
      \markboth{ Gelu Popescu}{  Noncommutative  hyperbolic geometry
      on the unit
      ball of $B(\cH)^n$    }
      %\pagestyle{plain}
      %\begin{flushright}
       % \it Date of this draft: \today
      %\end{flushright}
      %\bigskip

%%possible title:  Noncommutative extension of the Poincar\' e--Bergman    metric

      \title [   Noncommutative  hyperbolic geometry  on the unit
      ball of $B(\cH)^n$  ]
      {            Noncommutative  hyperbolic geometry   on the unit
      ball of $B(\cH)^n$
      }
        \author{Gelu Popescu}
      %\date{\today}
\date{June 3, 2008}
      \thanks{Research supported in part by an NSF grant}

      \address{Department of Mathematics, The University of Texas
      at San Antonio \\ San Antonio, TX 78249, USA}
      \email{\tt gelu.popescu@utsa.edu}

\begin{abstract}
In this paper we introduce a  hyperbolic (Poincar\' e-Bergman type)
distance $\delta$ on the noncommutative open ball
$$
[B(\cH)^n]_1:=\left\{(X_1,\ldots, X_n)\in B(\cH)^n:\ \|X_1
X_1^*+\cdots +X_nX_n^* \|^{1/2} <1\right\},
$$
where $B(\cH)$ is the algebra of all bounded linear operators on a
Hilbert space $\cH$.  It is proved that $\delta$ is  invariant under
the action of  the free holomorphic automorphism group of
$[B(\cH)^n]_1$, i.e.,
$$
\delta(\Psi(X), \Psi (Y))=\delta(X,Y),\quad X,Y\in
[B(\cH)^n]_1,
$$
for all $\Psi \in Aut([B(\cH)^n]_1)$. Moreover, we show that the
$\delta$-topology and the usual   operator norm topology coincide on
$[B(\cH)^n]_1$. While the open ball $[B(\cH)^n]_1$ is not a complete
metric space with respect to the operator norm topology,  we prove
that $[B(\cH)^n]_1$ is a complete metric space with respect to the
hyperbolic metric $\delta$. We obtain  an explicit formula for
$\delta$ in terms of the reconstruction operator
$$
R_X:=X_1^*\otimes R_1+\cdots +X_n^*\otimes R_n, \quad
X:=(X_1,\ldots, X_n)\in [B(\cH)^n]_1,
$$
associated with the right creation operators $R_1,\ldots, R_n$ on
the full Fock space with $n$ generators. In the particular case when
$\cH=\CC$, we show that the hyperbolic distance $\delta$ coincides
with the Poincar\' e-Bergman distance on the open unit ball
$$\BB_n:=\{z=(z_1,\ldots, z_n)\in\CC^n:\ \|z\|_2<1\}.
$$

We  obtain  a Schwarz-Pick lemma for free holomorphic functions on
$[B(\cH)^n]_1$   with respect to the hyperbolic metric, i.e., if
 $F:=(F_1,\ldots, F_m)$ is
a contractive ($\|F\|_\infty\leq 1$)  free holomorphic function,
then
$$
\delta(F(X), F(Y))\leq \delta(X,Y), \quad X,Y\in [B(\cH)^n]_1.
$$

The results of this paper are presented in the more general context
of Harnack parts of the closed ball $[B(\cH)^n]_1^-$, which are
noncommutative analogues of the Gleason parts of the Gelfand
spectrum of a function algebra.
\end{abstract}

      \maketitle

\bigskip

\bigskip

\section*{Introduction}

Poincar\' e's  discovery of a conformally invariant metric on the
open unit disc $\DD:=\{z\in \CC: \ |z|<1\}$  of the complex plane
was a cornerstone in the development of complex function theory. The
hyperbolic (Poincar\' e)  distance is defined on  $\DD$ by
$$
\delta_P(z,w):=\tanh^{-1}\left| \frac{z-w}{1-\bar z w}\right|,\quad
z,w\in \DD.
$$
Some of the basic and most important properties of the Poincar\' e
distance are the following:

\begin{enumerate}
\item
the Poincar\' e distance is invariant under the conformal
automorphisms of $\DD$, i.e.,
$$
\delta_P(\varphi(z), \varphi(w))=\delta_P(z,w),\quad z,w\in \DD,
$$
for all $\varphi\in Aut(\DD)$;
\item the $\delta_P$-topology induced on the open disc is the usual
planar topology;
\item
$(\DD, \delta_P)$ is a complete metric space;
\item  any analytic function $f:\DD\to \DD$ is distance-decreasing,
i.e., satisfies
$$
\delta_P(f(z), f(w))\leq \delta_P(z,w),\quad z,w\in \DD.
$$
\end{enumerate}

Bergman (see \cite{Be}) introduced an analogue of the Poincar\' e
distance for the open unit ball of $\CC^n$,
$$\BB_n:=\{z=(z_1,\ldots, z_n)\in \CC^n:\ \|z\|_2<1\},$$
  which  is defined by
 $$
\beta_n(z,w)=\frac{1}{2}\ln
\frac{1+\|\psi_z(w)\|_2}{1-\|\psi_z(w)\|_2},\qquad z,w\in \BB_n,
$$
where $\psi_z$ is the involutive automorphism of $\BB_n$ that
interchanges $0$ and $z$. The  Poincar\' e-Bergman distance has
 properties similar to those of $\delta_P$ (see (1)--(4)).  There is
a large literature concerning invariant metrics, hyperbolic
manifolds, and  the geometric viewpoint of complex function theory
(see \cite{Ko1}, \cite{Ko2}, \cite{Zhu}, and \cite{Kr} and the
references there in).

There are  several extensions  of the   Poincar\' e-Bergman distance
and related topics to more general domains.
 We mention the work of Suciu
(\cite{Su}, \cite{Su2}, \cite{Su3}), Foia\c s (\cite{Fo}), and And\^
o-Suciu-Timotin (\cite{AST}) on Harnack parts   of contractions and
Harnack  type distances between two contractions on Hilbert spaces.
Some of their results will be recover (with a different proof) in
the present paper, in the particular case when $n=1$.

In this paper,  we continue  our program to develop  a {\it
noncommutative function theory} on the unit ball of $B(\cH)^n$ (see
\cite{Po-holomorphic}, \cite{Po-free-hol-interp},
\cite{Po-pluriharmonic}, \cite{Po-majorants}, and
\cite{Po-automorphism}). The main goal is to introduce a hyperbolic
metric  $\delta$ on the noncommutative ball
$$
[B(\cH)^n]_1:=\left\{(X_1,\ldots, X_n)\in B(\cH)^n:\ \|X_1
X_1^*+\cdots +X_nX_n^* \|^{1/2} <1\right\},
$$
where $B(\cH)$ denotes  the algebra
  of all bounded linear operators on a Hilbert space $\cH$, which
  satisfy  properties similar to those of the Poincar\' e metric $\delta_P$ (see
  (1)--(3)), and which is a noncommutative extension  of the Poincar\'
  e-Bergman metric $\beta_n$ on the open unit ball of $\CC^n$.
The secondary goal is to obtain a Schwarz-Pick lemma for free
holomorphic functions on $[B(\cH)^n]_1$   with respect to the
hyperbolic metric.

We should mention that the noncommutative ball $[B(\cH)^n]_1$ can be
identified with the open unit ball of $B(\cH^n,\cH)$, which is one
of the infinite-dimensional Cartan domains studied  by L. Harris
(\cite{Ha1}, \cite{Ha2}, \cite{Ha3}). He has obtained several
results, related to our topic, in the more general setting of
$JB^*$-algebras (see also the book by H.~Upmeier \cite{Up}). We also
remark that the group of all free holomorphic automorphisms of
$[B(\cH)^n]_1$ (\cite{Po-automorphism}), can be identified with  a
subgroup of the group of automorphisms of $[B(\cH^n,\cH)]_1$
considered by R.S.~Phillips  \cite{Ph} (see also \cite{Y}). However,
the hyperbolic metric $\delta$  that we introduce in this paper is
different from the  Kobayahi metric  on  $[B(\cH)^n]_1$  and also
different  from  the metric  considered, for example,  in
\cite{Ha1}.

 In
\cite{Po-holomorphic}, \cite{Po-free-hol-interp},
\cite{Po-pluriharmonic}, \cite{Po-majorants}, and
\cite{Po-automorphism}, we obtained several results concerning   the
theory of free holomorphic (resp. pluriharmonic) functions  on
$[B(\cH)^n]_1$ and provided a framework for the study of arbitrary
 $n$-tuples of operators on a Hilbert space $\cH$. Several classical
 results from complex analysis (\cite{Co}, \cite{Ga}, \cite{H}, \cite{Ru})  have
 free analogues in
 the  noncommutative multivariable setting.
To put our work in perspective, we need   to set up  some notation
and recall some definitions.

 Let $\FF_n^+$ be the unital free semigroup on $n$ generators
$g_1,\ldots, g_n$ and the identity $g_0$.  The length of $\alpha\in
\FF_n^+$ is defined by $|\alpha|:=0$ if $\alpha=g_0$ and
$|\alpha|:=k$ if
 $\alpha=g_{i_1}\cdots g_{i_k}$, where $i_1,\ldots, i_k\in \{1,\ldots, n\}$.
If $(X_1,\ldots, X_n)\in B(\cH)^n$, we set $X_\alpha:= X_{i_1}\cdots
X_{i_k}$ and $X_{g_0}:=I_\cH$. Throughout this paper, we assume that
 $\cE$ is a  separable Hilbert space.
A map $F:[B(\cH)^n]_{1}\to B( \cH)\otimes_{min}B(\cE)$ is called
  {\it free
holomorphic function} on  $[B(\cH)^n]_{1}$  with coefficients in
$B(\cE)$ if there exist $A_{(\alpha)}\in B(\cE)$, $\alpha\in
\FF_n^+$, such that $\limsup_{k\to \infty} \left\| \sum_{|\alpha|=k}
A_{(\alpha)}^* A_{\alpha)}\right\|^{1/2k}\leq 1$ and
$$
F(X_1,\ldots, X_n)=\sum\limits_{k=0}^\infty \sum\limits_{|\alpha|=k}
X_\alpha\otimes  A_{(\alpha)},
$$
where the series converges in the operator  norm topology  for any
$(X_1,\ldots, X_n)\in [B(\cH)^n]_{1}$.
The set of all free holomorphic functions on $[B(\cH)^n]_1$ with
coefficients in $B(\cE)$ is denoted by $Hol(B(\cH)^n_1)$.
Let $H^\infty(B(\cH)^n_1)$  denote the set of  all elements $F$ in
$Hol(B(\cH)^n_1)$  such that
$$
\|F\|_\infty:=\sup  \|F(X_1,\ldots, X_n)\|<\infty,
$$
where the supremum is taken over all $n$-tuples  of operators
$(X_1,\ldots, X_n)\in [B(\cH)^n]_1$ and any Hilbert space $\cH$.
According to \cite{Po-holomorphic} and \cite{Po-pluriharmonic},
$H^\infty(B(\cH)^n_1)$ can be identified to the operator algebra
$ F_n^\infty\bar \otimes B(\cE)$ (the weakly closed algebra generated
by the spatial tensor product), where $F_n^\infty$ is the
noncommutative analytic Toeplitz algebra (see \cite{Po-von},
\cite{Po-multi}, \cite{Po-analytic}).

We say that a map
$u:[B(\cH)^n]_1\to B(\cH)\otimes_{min} B(\cE)$ is a self-adjoint
free pluriharmonic function on $[B(\cK)^n]_1$ if $u=\text{\rm Re}\,
f:=\frac{1}{2}(f^*+f)$ for some free holomorphic function $f$.  We
also recall that $u$ is called positive if $u(X_1,\ldots,X_n)\geq 0$
for any $(X_1,\ldots, X_n)\in [B(\cK)^n]_1$, where $\cK$ is an
infinite dimensional Hilbert space.

In Section 1, we introduce an equivalence relation $\overset{H}\sim$
on the closed ball $[B(\cH)^n]_1^-$,  and study  the equivalence
classes  (called Harnack parts ) with respect to $\overset{H}\sim$.
Two $n$-tuples  of operators $A:=(A_1,\ldots, A_n)$ and
$B:=(B_1,\ldots, B_n)$ in $[B(\cH)^n]_1^-$ are called Harnack
equivalent (and  denote $A\overset{H}{\sim}\, B$)  if and only if
there exists  a constant $c\geq 1$ such that
\begin{equation*}
 \frac{1}{c^2}\text{\rm Re}\,p(B_1,\ldots, B_n)\leq
\text{\rm Re}\,p(A_1,\ldots, A_n)\leq c^2 \text{\rm
Re}\,p(B_1,\ldots, B_n)
\end{equation*}
for any noncommutative polynomial $p\in \CC[X_1,\ldots, X_n]\otimes
M_{m}$, $m\in \NN$, with matrix-valued coefficients such that
$\text{\rm Re}\,p \geq 0$. Here $M_m$ denotes the algebra of all
$m\times m$  matrices with entries in $\CC$. We also use the
notation $A\overset{H}{{\underset{c}\sim}}\, B$ to emphasize the
constant $c$ in the inequalities above. The  Harnack parts of
$[B(\cH)^n]_1^-$ are noncommutative analogues of the Gleason parts
of the Gelfand spectrum of a function algebra (see \cite{G}).

 In Section 1, we use several results  (see \cite{Po-poisson},
 \cite{Po-curvature}, \cite{Po-varieties},  \cite{Po-unitary}) concerning  the theory of
 noncommutative Poisson transforms on Cuntz-Toeplitz $C^*$-algebras  (see \cite{Cu})
   and free pluriharmonic functions
  (see \cite{Po-pluriharmonic}, \cite{Po-majorants}) to obtain
    useful  characterizations for  the Harnack equivalence
 on the closed ball $[B(\cH)^n]_1^-$.  On the other hand, a characterization of positive
 free pluriharmonic functions (see \cite{Po-pluriharmonic}) and dilation theory (see \cite{SzF-book})
  are used
 to obtain  a Harnack type inequality (see \cite{Co}) for positive free
 pluriharmonic
 function on
$[B(\cH)^n]_1$. More precisely, we show that if $u$ is a  positive
free pluriharmonic function on $[B(\cH)^n]_1$ with operator-valued
coefficients in $B(\cE)$ and $0<r<1$, then
$$
  u(0) \,\frac{1-r}{1+r}\leq u(X_1,\ldots, X_n) \leq  u(0)
   \,\frac{1+r}{1-r}
$$
  for any  $ (X_1,\ldots, X_n)\in [B(\cH)^n]_r^-.$
This result  is crucial in order  to  prove that the
  open unit ball $[B(\cH)^n]_1$ is a distinguished Harnack part of
  $[ B(\cH)^n]_1^-$, namely, the Harnack part of $0$.

In Section 2, we introduce  a hyperbolic   ({\it Poincar\'e-Bergman}
type) metric
  on the Harnack parts of $[B(\cH)^n]_1^-$.
More precisely,   given  a Harnack part  $\Delta$ of
$[B(\cH)^n]_1^-$  we define $\delta:\Delta\times \Delta \to \RR^+$
by setting
\begin{equation*}
 \delta(A,B):=\ln \omega(A,B),\quad A,B\in \Delta,
\end{equation*}
where
\begin{equation*}
 \omega(A,B):=\inf\left\{ c> 1: \
A\,\overset{H}{{\underset{c}\sim}}\, B   \right\}.
\end{equation*}
We prove that  $\delta$ is a metric on  $\Delta$.

Consider the particular case when $\Delta=[B(\cH)^n]_1$ and let
 $\delta :[B(\cH)^n]_1\times [B(\cH)^n]_1\to [0,\infty)$ be
the hyperbolic metric defined   above.
We prove, in Section 2,  that $\delta$ is  invariant under
the action of the group $Aut([B(\cH)^n]_1)$  of all the free holomorphic
automorphisms of the noncommutative ball $[B(\cH)^n]_1$, i.e.,
$$\delta(\Psi(A), \Psi(B))=\delta(A,B),\qquad A,B\in [B(\cH)^n]_1,
$$
for all $\Psi\in Aut([B(\cH)^n]_1)$. We mention that the group
$Aut([B(\cH)^n]_1)$
 was determined in \cite{Po-automorphism}.

Using a characterization of  the Harnack equivalence  on $[B(\cH)^n]_1^-$ in terms of free
pluriharmonic kernels, we  obtain  an explicit formula for the hyperbolic distance
in terms of the reconstruction operator. More precisely, we show that
\begin{equation*}
\delta(A,B)=\ln \max \left\{ \left\|C_{A} C_{B}^{-1} \right\|,
  \left\|C_{B} C_{A}^{-1} \right\|\right\},\quad A, B\in [B(\cH)^n]_1,
\end{equation*}
where $C_X:=( \Delta_X\otimes I)(I-R_X)^{-1}$ and $R_X:=X_1^*\otimes
R_1+\cdots + X_n^*\otimes R_n$ is the reconstruction operator
associated with  the $n$-tuple $X:=(X_1,\ldots, X_n)\in [B(\cH)^n]_1$ and with
the right creation operators $R_1,\ldots, R_n$ on
the full Fock space  with $n$ generators.
In particular, we show that
$\delta|_{\BB_n\times \BB_n}$ coincides with the Poincar\'e-Bergman
distance on $\BB_n$, i.e.,
$$
\delta(z,w)=\frac{1}{2}\ln
\frac{1+\|\psi_z(w)\|_2}{1-\|\psi_z(w)\|_2},\qquad z,w\in \BB_n,
$$
where $\psi_z$ is the involutive automorphism of $\BB_n$ that
interchanges $0$ and $z$.
We mention that   similar  results  concerning the invariance under the automorphism group
$Aut([B(\cH)^n]_1)$ as well as  an explicit formula for the hyperbolic metric hold  on
any Harnack part
of $[B(\cH)^n]_1^-$.

In Section 3,  we study the relations between the $\delta$-topology,
the $d_H$-topology (which will be introduced), and the operator norm
topology on  Harnack parts of $[B(\cH)^n]_1^-$. We prove that the
hyperbolic metric  $\delta$ is a complete metric on any Harnack part
of $$[B_0(\cH)^n]_1:=\left\{ (X_1,\ldots, X_n)\in [B(\cH)^n]_1^-:\
r(X_1,\ldots, X_n)<1\right\}, $$ and that all the topologies above
coincide on the open ball $[B(\cH)^n]_1$. In particular, we deduce
that $[B(\cH)^n]_1$ is a complete metric space with respect to the
hyperbolic metric $\delta$ and that the $\delta$-topology and the
usual   operator norm topology coincide on $[B(\cH)^n]_1$.

A very important property of the Poincar\' e-Bergman distance
$\beta_m:\BB_m\times \BB_m\to \RR^+$ is that any holomorphic function $f:\BB_n\to \BB_m$
is distance-decreasing, i.e.,
$$\beta_m(f(z), f(w))\leq \beta_n(z,w),\quad z,w\in \BB_n.
$$
  In Section 4, we  extend this result and prove a
Schwarz-Pick lemma for free holomorphic functions on $[B(\cH)^n]_1$
with operator-valued coefficients, with respect to the hyperbolic
metric on the noncommutative ball $[B(\cH)^n]_1$.

 More precisely,
let $F_j:[B(\cH)^n]_1\to B(\cH) \otimes_{min} B(\cE)$, $j=1,\ldots,
m$, be free holomorphic functions with coefficients in $B(\cE)$, and
assume that $F:=(F_1,\ldots, F_m)$ is a contractive
$(\|F\|_\infty\leq 1$) free holomorphic function. If $X,Y\in
[B(\cH)^n]_1$, then we prove that
 $F(X)\overset{H}{\sim}\, F(Y)$ and
$$
\delta(F(X), F(Y))\leq \delta(X,Y),
$$
where $\delta$ is the hyperbolic  metric  defined on the Harnack
parts of the noncommutative ball $[B(\cH)^n]_1^-$.

 The present paper makes
connections between noncommutative   function  theory (see
\cite{Po-poisson}, \cite{Po-holomorphic}, \cite{Po-pluriharmonic},
\cite{Po-automorphism}) and   classical results in hyperbolic
complex analysis and geometry (see \cite{Ko1}, \cite{Ko2},
\cite{Kr}, \cite{Co}, \cite{Ga},  \cite{H}, \cite{Ru}). In
particular, we obtain a new formula of the Poincar\' e-Bergman
metric on $\BB_n$  using Harnack inequalities for positive free
pluriharmonic functions on $[B(\cH)^n]_1$, as well as  a formula in
terms of the left creation operators on the the full Fock space with
$n$ generators.

It would be interesting to see if the results of this paper can be
extended to more general infinite-dimensional bounded domains such
as the $JB^*$-algebras of Harris \cite{Ha1}, the domains considered
by Phillips \cite{Ph}, or the noncommutative domains from
\cite{Po-domains}.  Since our results are based on  noncommutative
function theory, dilation and model theory for  row contractions, we
are   inclined to believe in a positive  answer at least  for the
latter domains. We would  also like to thank the referee for useful
comments on the results of the paper and for bringing to our
attention several references.

\bigskip

\section{Harnack equivalence on the closed unit ball $[B(\cH)^n]_1^-$}

In this section, we  introduce  a preorder relation
$\overset{H}{\prec} $ on the closed ball $[B(\cH)^n]_1^-$ and
provide
    several   characterizations. This preorder
induces an equivalence relation $\overset{H}\sim$ on
$[B(\cH)^n]_1^-$, whose equivalence classes are called Harnack
parts. Several   characterizations for  the  Harnack parts are
provided.
    We obtain a Harnack type inequality  for positive free
 pluriharmonic
 functions and  use it
  to  prove that the
  open unit ball $[B(\cH)^n]_1$ is a distinguished Harnack part of
  $[ B(\cH)^n]_1^-$, namely, the Harnack part of $0$.

 Let $H_n$ be an $n$-dimensional complex  Hilbert space with
orthonormal
      basis
      $e_1$, $e_2$, $\dots,e_n$, where $n=1,2,\dots$, or $n=\infty$.
       We consider the full Fock space  of $H_n$ defined by
      $$F^2(H_n):=\CC1\oplus \bigoplus_{k\geq 1} H_n^{\otimes k},$$
      where  $H_n^{\otimes k}$ is the (Hilbert)
      tensor product of $k$ copies of $H_n$.
      Define the left  (resp.~right) creation
      operators  $S_i$ (resp.~$R_i$), $i=1,\ldots,n$, acting on $F^2(H_n)$  by
      setting
      $$
       S_i\varphi:=e_i\otimes\varphi, \quad  \varphi\in F^2(H_n),
      $$
       (resp.~$
       R_i\varphi:=\varphi\otimes e_i, \quad  \varphi\in F^2(H_n).
      $)
The noncommutative disc algebra $\cA_n$ (resp.~$\cR_n$) is the norm
closed algebra generated by the left (resp.~right) creation
operators and the identity. The   noncommutative analytic Toeplitz
algebra $F_n^\infty$ (resp.~$\cR_n^\infty$)
 is the  weakly
closed version of $\cA_n$ (resp.~$\cR_n$). These algebras were
introduced in \cite{Po-von} in connection with a noncommutative von
Neumann  type inequality \cite{vN}, and have been studied in several papers
(see \cite{Po-multi}, \cite{Po-funct}, \cite{Po-analytic}, \cite{Po-disc},  \cite{DP2}, \cite{DP1}, and the references
there in).

We need to  recall from \cite{Po-poisson} a few facts about
noncommutative Poisson transforms associated with row contractions
$T:=(T_1,\ldots, T_n)\in [B(\cH)^n]_1^-$.

     Let
$\FF_n^+$ be the unital free semigroup on $n$ generators
      $g_1,\dots,g_n$, and the identity $g_0$.  We denote
$e_\alpha:= e_{i_1}\otimes\cdots \otimes  e_{i_k}$ and $e_{g_0}:=1$.
Note that $\{e_\alpha\}_{\alpha\in \FF_n^+}$ is an orthonormal basis
for $F^2(H_n)$. For  each $0<r\leq 1$, define the defect operator
$\Delta_{T,r}:=(I_\cH-r^2T_1T_1^*-\cdots -r^2 T_nT_n^*)^{1/2}$.
The noncommutative Poisson  kernel associated with $T$ is the family
of operators
$$
K_{T,r} :\cH\to  \overline{\Delta_{T,r}\cH} \otimes  F^2(H_n), \quad
0<r\leq 1,
$$
defined by
\begin{equation*}
K_{T,r}h:= \sum_{k=0}^\infty \sum_{|\alpha|=k} r^{|\alpha|}
\Delta_{T,r} T_\alpha^*h\otimes  e_\alpha,\quad h\in \cH.
\end{equation*}
When $r=1$, we denote $\Delta_T:=\Delta_{T,1}$ and $K_T:=K_{T,1}$.
The operators $K_{T,r}$ are isometries if $0<r<1$, and
$$
K_T^*K_T=I_\cH- \text{\rm SOT-}\lim_{k\to\infty} \sum_{|\alpha|=k}
T_\alpha T_\alpha^*.
$$
Thus $K_T$ is an isometry if and only if $T$ is a {\it pure} row
 contraction,
i.e., $ \text{\rm SOT-}\lim\limits_{k\to\infty} \sum_{|\alpha|=k}
T_\alpha T_\alpha^*=0.$ We denote by $C^*(S_1,\ldots, S_n)$ the
Cuntz-Toeplitz $C^*$-algebra generated by the left creation operators.
The noncommutative Poisson transform at
 $T:=(T_1,\ldots, T_n)\in [B(\cH)^n]_1^-$ is the unital completely contractive  linear map
 $P_T:C^*(S_1,\ldots, S_n)\to B(\cH)$ defined by
 \begin{equation*}
 P_T[f]:=\lim_{r\to 1} K_{T,r}^* (I_\cH \otimes f)K_{T,r}, \qquad f\in C^*(S_1,\ldots,
 S_n),
\end{equation*}
 where the limit exists in the norm topology of $B(\cH)$. Moreover, we have
 $$
 P_T[S_\alpha S_\beta^*]=T_\alpha T_\beta^*, \qquad \alpha,\beta\in \FF_n^+.
 $$
 When $T:=(T_1,\ldots, T_n)$  is a pure row contraction,
   we have $$P_T[f]=K_T^*(I_{\cD_{T}}\otimes f)K_T,
   $$
   where $\cD_T=\overline{\Delta_T \cH}$.
We refer to \cite{Po-poisson}, \cite{Po-curvature},  and
\cite{Po-unitary} for more on noncommutative Poisson transforms on
$C^*$-algebras generated by isometries. For basic results concerning  completely bounded maps
 and operator spaces we refer to \cite{Pa-book}, \cite{Pi}, and \cite{ER}.

When $T\in [B(\cH)^n]_1^-$ is a completely
non-coisometric (c.n.c.) row contraction, i.e.,
 there is no $h\in \cH$, $h\neq 0$, such that
 $$
 \sum_{|\alpha|=k}\|T_\alpha^* h\|^2=\|h\|^2
 \quad \text{\rm for any } \ k=1,2,\ldots,
 $$
an  $F_n^\infty$-functional calculus was developed  in
\cite{Po-funct}.
 We showed that if $f=\sum\limits_{\alpha\in \FF_n^+} a_\alpha S_\alpha$ is
 in $F_n^\infty$, then
 \begin{equation*}
 \Gamma_T(f)=f(T_1,\ldots, T_n):=
 \text{\rm SOT-}\lim_{r\to 1}\sum_{k=0}^\infty
  \sum_{|\alpha|=k} r^{|\alpha|} a_\alpha T_\alpha
\end{equation*}
exists and $\Gamma_T:F_n^\infty\to B(\cH)$ is a completely
contractive homomorphism and WOT-continuous (resp. SOT-continuous)
on bounded sets. Moreover, we showed (see \cite{Po-unitary})  that
$\Gamma_T(f)=P_T[f]$, $f\in F_n^\infty$, where
 \begin{equation}
 \label{pois-sot}
  P_T[f]:= \text{\rm
SOT-}\lim_{r\to 1}K_{T,r} ( I_\cH\otimes f) K_{T,r},\qquad f\in
F_n^\infty,
\end{equation} is the extension of the noncommutative Poisson
transform  to the noncommutative analytic Toeplitz algebra  $F_n^\infty$.

 We   introduced  in \cite{Po-pluriharmonic} the noncommutative Poisson
 transform $\cP\mu$
of a completely bounded linear map   $\mu:\cR_n^*+ \cR_n\to B(\cE)$
by setting
$$
(\cP\mu)(X_1,\ldots, X_n):=( \text{\rm id}\otimes
\mu)\left[P(X,R)\right], \qquad X:=(X_1,\ldots,X_n)\in [B(\cH)^n]_1,
$$
where the {\it free  pluriharmonic Poisson kernel} $P(X,R)$  is
given by
$$
P(X,R):=\sum_{k=1}^\infty\sum_{|\alpha|=k}
 X_\alpha^*\otimes R_{\widetilde\alpha}
+I+\sum_{k=1}^\infty\sum_{|\alpha|=k} X_\alpha\otimes
R_{\widetilde\alpha}^*,\qquad X\in [B(\cH)^n]_1,
$$
and the series are convergent in the operator norm topology.
  We recall that the  joint spectral radius associated
 with an  $n$-tuple of operators
$ (X_1,\ldots, X_n)\in B(\cH)^n$  is given by
$$
r(X_1,\ldots, X_n):=\lim_{k\to \infty}\left\|\sum_{|\alpha|=k}
X_\alpha X_\alpha^*\right\|^{1/2k}.
$$
We remark that  the free  pluriharmonic Poisson kernel $P(X,R)$
makes sense  for any   $n$-tuple of operators $X:=(X_1,\ldots, X_n)\in B(\cH)^n$
with $r(X_1,\ldots, X_n)<1$.
  According to \cite{Po-unitary}, $X\in [B(\cH)^n]_1^-$ if and only if
$$r(X_1,\ldots, X_n)\leq
1\quad \text{ and } \quad P(rX,R)\geq 0, \quad r\in [0,1). $$

We say that  a free pluriharmonic function $u$ is  positive
 if  $u(X_1,\ldots, X_n)\geq 0$ for any
$(X_1,\ldots, X_n)\in [B(\cK)^n]_\gamma$ and  any Hilbert space
$\cK$.  We recall \cite{Po-pluriharmonic}  that  $u\geq 0$ if and
only if $u(rS_1,\ldots, rS_n)\geq 0$ for any $r\in [0,1)$. In
particular, if $p$ is a noncommutative polynomial with
operator-valued coefficients, then $\text{\rm Re}\,p\geq 0$ if and
only if $\text{\rm Re}\,p(S_1,\ldots, S_n)\geq 0$.

Now, we  introduce  a preorder relation
$\overset{H}{\prec} $ on the closed ball $[B(\cH)^n]_1^-$ and
provide
    several   characterizations.
Let $A:=(A_1,\ldots, A_n)$ and $B:=(B_1,\ldots, B_n)$ be in
$[B(\cH)^n]_1^-$. We say that $A$ is Harnack dominated by $B$, and
denote $A\overset{H}{\prec}\, B$, if there exists $c>0$ such that
$$\text{\rm Re}\,p(A_1,\ldots, A_n)\leq c^2 \text{\rm
Re}\,p(B_1,\ldots, B_n) $$
 for any noncommutative polynomial with
matrix-valued coefficients $p\in \CC[X_1,\ldots, X_n]\otimes M_{m}$,
$m\in \NN$, such that $\text{\rm Re}\,p\geq 0$.
 When we
want to emphasize the constant $c$, we write
$A\overset{H}{{\underset{c}\prec}}\, B$.

\begin{theorem}
\label{equivalent} Let $A:=(A_1,\ldots, A_n)$ and $B:=(B_1,\ldots,
B_n)$ be in $[B(\cH)^n]_1^-$ and let $c>0$. Then the following
statements are equivalent:
\begin{enumerate}
\item[(i)]
$A\overset{H}{{\underset{c}\prec}}\, B$;
\item[(ii)] $ (P_A\otimes_{min}\text{\rm id})[q^*q]\leq c^2
 (P_B\otimes_{min}\text{\rm id})[q^*q]$ for any
polynomial $q=\sum_{|\alpha|\leq k} S_\alpha\otimes C_{(\alpha)}$
with matrix-valued coefficients $C_{(\alpha)}\in M_{m}$, and
$k,m\in \NN$, where $P_X$ is the noncommutative Poisson transform at $X\in
[B(\cH)^n]_1^-$;
\item[(iii)]
$P( rA, R)\leq c^2 P(rB, R)$ for any $r\in [0,1)$, where $P(X,R)$ is
the free pluriharmonic  Poisson kernel  associated with $X\in
[B(\cH)^n]_1$;
\item[(iv)]
$u(rA_1,\ldots, rA_n)\leq c^2 u(rB_1,\ldots, rB_n)$ for any positive
free pluriharmonic function $u$  with operator valued coefficients
and any $r\in [0,1)$;
\item[(v)]
$  c^2{P}_B-{P}_A$   is a completely positive linear map on the
operator space $\overline{\cA_n^*+ \cA_n}^{\|\cdot\|}$.
\end{enumerate}
\end{theorem}

\begin{proof} First we prove the implimation $(i)\leftrightarrow (ii)$.
Assume that (i) holds, and let $q=\sum_{|\alpha|\leq k} S_\alpha\otimes C_{(\alpha)}$
 be an arbitrary  polynomial in $\cA_n\otimes M_m$. Since the left creation operators are isometries with orthogonal ranges,
 $q^*q$ has
    form $\text{\rm Re}\,p(S_1,\ldots, S_n)$ for some
    noncommutative polynomial  $p\in \CC[X_1,\ldots, X_n]\otimes
M_{m}$, $m\in \NN$, such that $\text{\rm Re}\,p\geq 0$. Using the properties of the noncommutative
 Poisson transform, one can see that $(i)\implies (ii)$.
 Conversely, assume that (ii) holds. Let $p\in \CC[X_1,\ldots, X_n]\otimes
M_{m}$, $m\in \NN$, be a polynomial of degree $k$  such that
$\text{\rm Re}\,p\geq 0$. Then $\text{\rm Re}\,p(S_1,\ldots, S_n)$
is a positive multi-Toeplitz operator with respect to $R_1,\ldots,
R_n$
    acting on the Hilbert space
   $F^2(H_n)\otimes \CC^m$, i.e.,
   $$
   (R_i^*\otimes I_{\CC^m}) [\text{\rm Re}\,p(S_1,\ldots, S_n)](R_j\otimes I_{\CC^m})=\delta_{ij}G,\qquad
   i,j=1,\ldots, n.
   $$
    According
    to the Fej\' er type
   factorization theorem of \cite{Po-analytic}, there
   exists a multi-analytic operator
   $q$ in  $B(F^2(H_n)\otimes \CC^m)$ such that
   $\text{\rm Re}\,p(S_1,\ldots, S_n)= q^* q$ and
   $$
   (S_\alpha^*\otimes I_{\CC^m}) q (1\otimes h)=0
   $$
   for any $h\in \CC^m$ and $|\alpha|>k$.
   Therefore, $q$ is a polynomial, i.e.,
   $q=\sum\limits_{|\alpha|\leq k} S_\alpha \otimes C_{(\alpha)}$
   for some operators $C_{(\alpha)}\in B(\CC^m)$. Note that
   \begin{equation*}
   \begin{split}
\text{\rm Re}\,p(A_1,\ldots, A_n)&=
(P_A\otimes_{min}\text{\rm id})[q^*q]
\leq c^2
 (P_B\otimes_{min}\text{\rm id})[q^*q]
 = c^2 \text{\rm
Re}\,p(B_1,\ldots, B_n),
   \end{split}
   \end{equation*}
which proves  (i).

Let us prove that $(i)\implies (iii)$.
For each $m\in \NN$, consider
$R^{(m)}:=(R_1^{(m)},\ldots, R_n^{(m)})$, where   $R_i^{(m)}$,  $i=1,\ldots,
 n$, is the compression of the right creation operator $R_i$ to $\cP_m:=\text{\rm
 span}\,\{e_\alpha:\ \alpha\in \FF_n^+, |\alpha|\leq m\}$. Note that
 $R_\alpha^{(m)}=0$ for any $\alpha\in \FF_n^+$ with $|\alpha|\geq
 m+1$ and, consequently, we have
 $$
P( rX, R^{(m)})=\sum_{1\leq |\alpha|\leq m} r^{|\alpha|}
X_\alpha^*\otimes R_{\widetilde\alpha}^{(m)} +I +\sum_{1\leq
|\alpha|\leq m} r^{|\alpha|} X_\alpha \otimes
{R_{\widetilde\alpha}^{(m)}}^*.
$$
Note that $R^{(m)}$ is a pure row contraction and the noncommutative
Poisson transform $\text{\rm id}\otimes P_{R^{(m)}}$ is a completely positive map.
We recall that $X\mapsto P(X,R)$ is a positive free pluriharmonic
 function with coefficients in $B(F^2(H_n)$.
Hence  $P(rX, R)\geq 0$ for any $X\in [B(\cH)^n]_1^-$ and  $r\in
 [0,1)$.  Applying $\text{\rm id}\otimes P_{R^{(m)}}$, we obtain
 $$
P(rX, R^{(m)})=(\text{\rm id}\otimes
P_{R^{(m)}})\left[P(rX,S)\right]\geq 0.
$$
Now, applying (i), we obtain
$$
P(rA, R^{(m)})\leq c^2 P(rB, R^{(m)})$$
 for any $m\in \NN$ and  $r\in [0,1)$. Using  Lemma 8.1 from
 \cite{Po-pluriharmonic}, we deduce that
 $P(rA, R)\leq c^2 P(rB, R)$ for any $r\in [0,1)$. Therefore (iii)
 holds.

  To prove the implication $(iii)\implies (iv)$, assume that
 condition (iii) holds and let $u$ be a positive free pluriharmonic
 function with coefficients in $B(\cE)$. According to Corollary 5.5
 from \cite{Po-pluriharmonic}, there exists a completely positive
 linear map $\mu: \cR_n^* + \cR_n\to B(\cE)$ such that
 $$
 u(Y)=(\cP \mu)(Y):=( \text{\rm id}\otimes \mu)(P(Y,R))
 $$
 for any $Y\in [B(\cH)^n]_1$. Hence and using the fact that
 $c^2P(rB, R)-P(rA, R)\geq 0$, we deduce that
 $$
 c^2 u(rB_1,\ldots, rB_n)-u(rA_1,\ldots, rA_n)=( \text{\rm
 id}\otimes \mu)[c^2P(rB, R)-P(rA, R)]\geq 0,
$$
which proves (iv).

Now, we prove the implication $(iv)\implies (v)$. Let $g\in
\overline{\cA_n^*+ \cA_n}^{\|\cdot\|}\otimes M_{m}$ be positive.
Then, according to  Theorem 4.1 from \cite{Po-pluriharmonic},  the
map defined by
  \begin{equation}\label{u}
  u(X):=(P_X\otimes \text{\rm id})[g],\qquad X\in [B(\cH)^n]_1,
  \end{equation}
   is a positive  free
  pluriharmonic function. Condition (iv) implies
  $u(rA_1,\ldots, rA_n)\leq c^2 u(rB_1,\ldots, rB_n)$ for   any $r\in
  [0,1)$. On the other hand, by relation \eqref{u}, we have
\begin{equation}
\label{cc}
c^2(P_{rB}\otimes \text{\rm id})[ g]-(P_{rA}\otimes \text{\rm id})[
g]=c^2 u(rB_1,\ldots, rB_n)- u(rA_1,\ldots, rA_n)\geq 0
\end{equation}
for any $r\in [0,1)$. Since $g\in \overline{\cA_n^*+
\cA_n}^{\|\cdot\|}\otimes M_{m}$, we have
\begin{equation*}
(P_{A}\otimes \text{\rm id})[ g]=\lim_{r\to 1}
(P_{rA}\otimes \text{\rm id})[ g]\quad \text{ and }
 \quad (P_{B}\otimes \text{\rm id})[ g]=\lim_{r\to 1}
(P_{rB}\otimes \text{\rm id})[ g]
\end{equation*}
where the convergence is in the operator norm topology. Taking $r\to
1$ in \eqref{cc}, we deduce item (v). To prove the implication
$(v)\implies (i)$, let $p\in \CC[X_1,\ldots, X_n]\otimes M_{m}$,
$m\in \NN$, be a noncommutative polynomial with matrix coefficients
such that $\text{\rm Re}\,p\geq 0$.  Due to the proprieties of the
noncommutative Poisson transform, we have
$$
c^2 \text{\rm Re}\,p(B_1,\ldots, B_n)- \text{\rm Re}\,p(A_1,\ldots,
A_n)=c^2(P_B\otimes \text{\rm id})[\text{\rm
Re}\,p(S_1,\ldots,S_n)]-(P_A\otimes \text{\rm id})[\text{\rm
Re}\,p(S_1,\ldots,S_n)].
$$
Since $\text{\rm Re}\,p(S_1,\ldots, S_n)\geq 0$ and  $
c^2{P}_B-{P}_A$   is a completely positive linear map on the
operator space $\overline{\cA_n^*+ \cA_n}^{\|\cdot\|}$, we deduce
item (i). This completes the proof.
\end{proof}

We remark that each item in Theorem \ref{equivalent} is equivalent
to  the following: $$ P( rA, R^{(m)})\leq c^2 P( rB, R^{(m)})\quad
\text{
 for any }\ m\in \NN, r\in [0,1),
 $$
where  $R^{(m)}$ is defined in the proof of Theorem \ref{equivalent}.

In  what follows, we characterize  the elements of the closed ball $[B(\cH)^n]_1^-$
which are Harnack dominated  by $0$.

\begin{theorem}
\label{A<0} Let $A:=(A_1,\ldots, A_n)$   be in $[B(\cH)^n]_1^-$.
Then $A\overset{H}{{ \prec}}\, 0$ if and only if the joint spectral
radius  $r(A_1,\ldots, A_n)<1$.
\end{theorem}
\begin{proof}
Note that the map $X\mapsto P(X,R)$ is a positive  free pluriharmonic function on
$[B(\cH)^n]_1$ with coefficients in
 $B(F^2(H_n))$ and has the factorization
\begin{equation*}
\begin{split}
P(X,R)&= (I-R_X)^{-1}-I+ (I-R_X^*)^{-1}\\
&=(I-R_X^*)^{-1}\left[ I-R_X-(I-R_X^*)(I-R_X)+ I-R_X^*
\right](I-R_X)^{-1}\\
&= (I-R_X^*)^{-1}\left[  (I-X_1X_1^*-\cdots -X_nX_n^*)\otimes I
\right](I-R_X)^{-1},
\end{split}
\end{equation*}
where $R_X:=X_1^*\otimes R_1+\cdots +X_n^*\otimes R_n$ is the
reconstruction operator associated with  the $n$-tuple
$X:=(X_1,\ldots, X_n)\in[B(\cH)^n]_1$. We remark that,  due to the
fact that the spectral radius of $R_X$ is equal
 to $r(X_1,\ldots, X_n)$,
the factorization above  holds  for any $X\in [B(\cH)^n]_1^-$
with $r(X_1,\ldots, X_n)<1$.

Now, using
 Theorem \ref{equivalent} part (iii) and the factorization obtained above,
 we deduce that $A\overset{H}{{ \prec}}\, 0$
if and only if  there exists $c\geq 1$ such that
$$
(I-rR_A^*)^{-1}\left[  (I-r^2A_1A_1^*-\cdots - r^2 A_n A_n^*)\otimes
I\right] (I-rR_A)^{-1}\leq c^2I
$$
for any $r\in [0,1)$. Similar inequality holds if we replace the
right creation operators by the left creation operators. Then,
applying the noncommutative Poisson transform $ \text{\rm id}\otimes
P_{e^{i\theta}R }$   we obtain
 \begin{equation}
\label{equi3} (I-r^2A_1A_1^*-\cdots - r^2 A_n A_n^*)\otimes I\leq
c^2(I-re^{-i\theta}R_A^*)(I-re^{i\theta}R_A)
\end{equation}
for any $r\in [0,1)$ and $\theta\in  \RR$.

 Assume now that
$A:=(A_1,\ldots, A_n)\in [B(\cH)^n]_1^-$ is such that
$A\overset{H}{{ \prec}}\, 0$. Then $r(A_1,\ldots, A_n)\leq 1$.
Suppose that $r(A_1,\ldots, A_n)= 1$. Since $r(R_A)=r(A_1,\ldots,
A_n)$, there exists $\lambda_0\in \TT$ in the approximative spectrum
of $R_A$. Consequently, there is a sequence $\{h_m\}$ in $\cH\otimes
F^2(H_n)$ such that $\|h_m\|=1$ and $\|\lambda_0h_m-R_Ah_m\|\leq
\frac{1}{m}$ for $m=1,2,\ldots$. Hence and taking $r=1-\frac{1}{m}$
in relation \eqref{equi3}, we deduce that
\begin{equation}
\begin{split}
\left< \left(I-\left(1-\frac{1}{m}\right)^2 R_A^* R_A\right)h_m,
h_m\right> &\leq c^2\left\| h_m-\left(
1-\frac{1}{m}\right)\bar\lambda_0 R_A h_m\right\|^2\\
&\leq c^2\left( \left\|\lambda_0h_m-R_A h_m\right\|+ \frac{1}{m}
\|R_A h_m\|\right)^2 \leq \frac{4c^2}{m^2}.
\end{split}
\end{equation}
 Combining this result with the fact that
$\|R_Ah_m\|\leq 1$, we deduce that
$$
1-\left(1-\frac{1}{m}\right)^2\leq
1-\left(1-\frac{1}{m}\right)^2\|R_Ah_m\|^2\leq \frac{4c^2}{m^2}.
$$
 Hence, we obtain  $2m\leq 4c^2 +1$ for any
$m=1,2,\ldots$, which is a contradiction. Therefore, we have
$r(A_1,\ldots, A_n)< 1$.

Conversely, assume that  $A:=(A_1,\ldots, A_n)\in [B(\cH)^n]_1^-$
has the joint spectral radius  $r(A_1,\ldots, A_n)< 1$.  Note that $
M:=\sup_{r\in (0,1)} \|(I-rR_A)^{-1}\| $ exists and, therefore,
$$
(I-rR_A^*)^{-1}\left[ (I-r^2A_1A_1^*-\cdots - r^2 A_n A_n^*)\otimes
I\right] (I-rR_A)^{-1}\leq M^2I
$$
for any $r\in (0,1)$, which, due to Theorem \ref{equivalent},
 shows that $A\overset{H}{{ \prec}}\, 0$.
The proof is complete.
\end{proof}
We mention that in the particular case  when $n=1$ we can recover a
result obtained
 in \cite{AST}.

   Since $\overset{H}{\prec} $ is a preorder
relation on $[B(\cH)^n]_1^-$, it induces an equivalent relation
$\overset{H}\sim$ on $[B(\cH)^n]_1^-$, which we call Harnack
equivalence. The equivalence classes with respect to
$\overset{H}\sim$ are called Harnack parts of $[B(\cH)^n]_1^-$. Let
$A:=(A_1,\ldots, A_n)$ and $B:=(B_1,\ldots, B_n)$ be in
$[B(\cH)^n]_1^-$. It is easy to see that $A$ and $B$ are Harnack
equivalent (we denote $A\overset{H}{\sim}\, B$)  if and only if
there exists $c\geq  1$ such that
\begin{equation}
\label{<<}
 \frac{1}{c^2}\text{\rm Re}\,p(B_1,\ldots, B_n)\leq
\text{\rm Re}\,p(A_1,\ldots, A_n)\leq c^2 \text{\rm
Re}\,p(B_1,\ldots, B_n)
\end{equation}
for any noncommutative polynomial with matrix-valued coefficients
$p\in \CC[X_1,\ldots, X_n]\otimes M_{m}$, $m\in \NN$, such
that $\text{\rm Re}\,p\geq 0$. We also use the
notation $A\overset{H}{{\underset{c}\sim}}\, B$ if
$A\overset{H}{{\underset{c}\prec}}\, B$ and
$B\overset{H}{{\underset{c}\prec}}\, A$.

 A completely positive  (c.p.) linear map  $\mu_X:C^*(S_1,\ldots, S_n)\to
 B(\cH)$ is called  {\it representing c.p. map} for the point
 $X:=(X_1,\ldots, X_n)\in [B(\cH)^n]_1^-$ if
 $$
 \mu(S_\alpha)=X_\alpha   \quad \text{ for any } \ \alpha\in
 \FF_n^+.
 $$

  Next, we obtain several characterizations for the Harnack equivalence on the closed ball
$[B(\cH)^n]_1^-$. The result will play a crucial role in this paper.

\begin{theorem}
\label{equivalent2} Let $A:=(A_1,\ldots, A_n)$ and $B:=(B_1,\ldots,
B_n)$ be in $[B(\cH)^n]_1^-$ and let $c>1$. Then the following
statements are equivalent:
\begin{enumerate}
\item[(i)]
$A\overset{H}{{\underset{c}\sim}}\, B$;
\item[(ii)] the noncommutative Poisson transform satisfies  the inequalities
$$ \frac{1}{c^2}(P_B\otimes_{min}\text{\rm id})[q^*q]\leq
(P_A\otimes_{min}\text{\rm id})[q^*q]\leq c^2
(P_B\otimes_{min}\text{\rm id})[q^*q]$$ for any polynomial
$q=\sum_{|\alpha|\leq k} S_\alpha\otimes A_{(\alpha)}$ with
matrix-valued coefficients $A_{(\alpha)}\in M_{m}$, and
$k,m\in \NN$;
\item[(iii)]  the free pluriharmonic kernel satisfies the inequalities
$$\frac{1}{c^2} P(rB,R)\leq P( rA, R)\leq c^2 P(rB, R)$$
for any $r\in [0,1)$;
\item[(iv)]for any positive free pluriharmonic function
$u$  with operator valued coefficients and any $r\in [0,1)$,
$$\frac{1}{c^2}u(rB_1,\ldots, rB_n)\leq  u(rA_1,\ldots, rA_n)\leq c^2
u(rB_1,\ldots, rB_n);$$
\item[(v)]
$  c^2{P}_B-{P}_A$  and  $  c^2{P}_A-{P}_B$ is a completely positive
linear map on the operator space $\overline{\cA_n^*+
\cA_n}^{\|\cdot\|}$, where ${ P}_X$ is the noncommutative Poisson
transform at $X\in [B(\cH)^n]_1^-$;
\item[(vi)] there are   representing c.p. maps
$\mu_A$ and $\mu_B$ for $A$ and $B$, respectively, such that
$$
\frac{1}{c^2} \mu_B\leq \mu_A\leq c^2 \mu_B.
$$
\end{enumerate}
\end{theorem}

\begin{proof} The first five equivalences follow using Theorem
\ref{equivalent}. It remains to show that $(i)\leftrightarrow (vi)$.
Assume that $A\overset{H}{{\underset{c}\sim}}\, B$. Then according
to  item (v), $P_A-\frac{1}{c^2} P_B$ and $P_B-\frac{1}{c^2} P_A$
are completely positive linear maps on the operator space
 $\overline{\cA_n^*+ \cA_n}^{\|\cdot\|}$. Using Arveson's extension
 theorem, we find  some  completely positive linear maps $\varphi$ and $\psi$  on
 the Cuntz-Toeplitz algebra
 $C^*(S_1,\ldots, S_n)$ which are extensions of $P_A-\frac{1}{c^2} P_B$ and $P_B-\frac{1}{c^2}
P_A$, respectively. Define the c.p. maps $\mu_A,
\mu_B:C^*(S_1,\ldots, S_n) \to B(\cH)$ by setting
\begin{equation}
\label{muab}
 \mu_A:=\frac{c^2}{c^4-1} (c^2\varphi+ \psi)\quad \text{
and } \quad \mu_B:=\frac{c^2}{c^4-1} (c^2\psi+ \varphi).
\end{equation}
Note that for any $f\in \overline{\cA_n^*+ \cA_n}^{\|\cdot\|}$, we
have
$$
P_A[f]=\varphi(f)+\frac{1}{c^2}
P_B[f]=\varphi(f)+\frac{1}{c^2}\left[
\psi(f)+\frac{1}{c^2}P_A\right].
$$
Solving for $P_A[f]$, we obtain $P_A[f]=\mu_A(f)$. Similarly, we
obtain $P_B[f]=\mu_B(f)$. Since $P_A[S_\alpha]=A_\alpha $ and
$P_B[S_\alpha]=B_\alpha$ for any $\alpha\in \FF_n^+$, it is clear
that $\mu_A$, $\mu_B$ are representing c.p. maps  for
$A:=(A_1,\ldots, A_n)$ and $B:=(B_1,\ldots, B_n)$, respectively.
Now, since $c>1$ and using  relation \eqref{muab},   it is a routine  to show
that
\begin{equation}\label{bab}
 \frac{1}{c^2} \mu_B\leq \mu_A\leq c^2 \mu_B.
\end{equation}

Conversely, assume that  (vi) holds for some $c>1$, and let
$p(S_1,\ldots, S_n):=\sum_{|\alpha|\leq q} S_\alpha\otimes
M_{(\alpha)}$ be a polynomial such that $\text{\rm
Re}\,p\geq 0$. Since  $\mu_A$, $\mu_B$ are representing c.p. maps  for
$A:=(A_1,\ldots, A_n)$ and $B:=(B_1,\ldots, B_n)$, respectively, relation
\eqref{bab} implies
$$
\frac{1}{c^2}\text{\rm Re}\,p(B_1,\ldots, B_n)\leq \text{\rm
Re}\,p(A_1,\ldots, A_n)\leq c^2 \text{\rm Re}\,p(B_1,\ldots, B_n).
$$
 This shows  that
$A\overset{H}{{\underset{c}\sim}}\, B$ and  completes the proof.
\end{proof}

We remark that the first five equivalences  in  Theorem
\ref{equivalent2} remain true even when $c\geq 1$.

The  next result is a  Harnack  type inequality for positive free pluriharmonic functions on
 $[B(\cH)^n]_1$.

\begin{theorem}
\label{Harnack} If \,$u$ is a  positive free pluriharmonic function
on $[B(\cH)^n]_1$  with operator-valued  coefficients in $B(\cE)$
and $0\leq r<1$, then
$$
  u(0) \,\frac{1-r}{1+r}\leq u(X_1,\ldots, X_n) \leq  u(0)
   \,\frac{1+r}{1-r}
$$
  for any  $ (X_1,\ldots, X_n)\in [B(\cH)^n]_r^-.$
\end{theorem}
\begin{proof}
Notice that the free pluriharmonic  Poisson kernel satisfies the
relation
$$
P(rX,R)=\sum_{k=1}^\infty r^k R_X^k + I+ \sum_{k=1}^\infty r^k
(R_X^*)^k, \qquad 0\leq r<1,
$$
for any $X:=(X_1,\ldots, X_n)\in [B(\cH)^n]_1$, where $R_X:=
X_1^*\otimes R_1+\cdots +  X_n^*\otimes R_n$ is the reconstruction
operator. Since $R_1,\ldots, R_n$ are isometries with orthogonal
ranges, we have $R_X^* R_X= \sum_{i=1}^n X_iX_i^* \otimes I$ and
therefore $\|R_X\|=\left\|\sum_{i=1}^n X_iX_i^*\right\|^{1/2}<1$.
Let ${\bf U}$ be the minimal unitary dilation of $R_X$ on a Hilbert
space $\cM\supset  \cH\otimes F^2(H_n)$, in the spirit of
Sz.-Nagy--Foia\c s. Then we have $R_X^k=P_{\cH\otimes F^2(H_n)} {\bf
U}^k|_{\cH\otimes F^2(H_n)}$ for any $k=1,2,\ldots$. Since $R_X$ is
a strict contraction, ${\bf U}$ is a bilateral shift which can be
identified with $M_{e^{i\theta}}\otimes I_\cG$ for some Hilbert
space $\cG$, where $M_{e^{i\theta}}$  is the multiplication operator
by $e^{i\theta}$ on $L^2(\TT)$. Consequently,  there is a unitary
operator $W:L^2(\TT)\otimes \cG\to \cM$ such that
\begin{equation}
\label{dil} P(rX,R)=P_{\cH\otimes F^2(H_n)} [W({\bf M}(r)\otimes
I_\cG) W^*]|_{\cH\otimes F^2(H_n)},
\end{equation}
where $$ {\bf M}(r):=\sum_{k=1}^\infty r^kM_{e^{i\theta}}^k+ I+
\sum_{k=1}^\infty r^kM_{e^{-i\theta}}^k.
$$
For each $f\in L^2(\TT)$, we have
$$
\left<{\bf M}(r) f,f\right>=\|f\|_2^2\sum_{k=-\infty}^{k=\infty}
r^{|k|} e^{i\theta k}=\|f\|_2^2\frac{1-r^2}{1-2r\cos \theta +r^2}.
$$
Since
$$\frac{1-r}{1+r}\leq \frac{1-r^2}{1-2r\cos \theta +r^2}\leq
\frac{1+r}{1-r},
$$
we have $\frac{1-r}{1+r} I\leq  {\bf M}(r)\leq \frac{1+r}{1-r}I$.
Hence and due to \eqref{dil}, we deduce that
\begin{equation}
\label{ineq-Pois} \frac{1-r}{1+r} I \leq P( rX, R)\leq
\frac{1+r}{1-r}I
\end{equation}
for any $X\in [B(\cH)^n]_1$ and $0\leq r<1$. Since  $Y\mapsto
P(Y,R)$ is a pluriharmonic function on $[B(\cH)^n]_1$, hence
continuous,  we deduce that \eqref{ineq-Pois} holds for any $X\in
[B(\cH)^n]_1^-$.

On the other hand, since $u$ is a positive pluriharmonic function on
$[B(\cH)^n]_1$,  one can use  \cite{Po-pluriharmonic} (see Corollary
5.5) to find a completely positive linear map $\mu:\cR_n^*+\cR_n\to
B(\cE)$ such that $u$ is the noncommutative  Poisson transform of
$\mu$, i.e.,
$$u(Y_1,\ldots, Y_n)=( \text{\rm id}\otimes \mu)[P(Y,R)],\qquad
Y:=(Y_1,\ldots, Y_n)\in [B(\cH)^n]_1.
$$
Hence, $u(rX_1,\ldots, rX_n)=(\text{\rm id}\otimes \mu)[P( rX, R)]$
for any $X:=(X_1,\ldots, X_n)\in [B(\cH)^n]_1^-$. Now, using
inequalities \eqref{ineq-Pois} and  the fact that $\mu$ is a
completely positive linear map, we obtain
$$
  u(0) \,\frac{1-r}{1+r}\leq u(rX_1,\ldots, rX_n) \leq  u(0)
   \,\frac{1+r}{1-r},\qquad X\in [B(\cH)^n]_1^-.
$$
This completes the proof.
\end{proof}

We recall that if $f\in \cA_n\otimes M_m$ and $(A_1,\ldots, A_n)\in
[B(\cH)^n]_1^-$, then, due to the noncommutative  von Neumann
inequality \cite{Po-von} (see also \cite{Po-funct}), it makes sense
to define   $f(A_1,\ldots, A_n)\in B(\cH)\otimes_{min}M_m$. In this
case we have $\|f(A_1,\ldots, A_n)\|\leq \|f\|. $

Our next task is to determine the Harnack part of $0$. First, we need the
 following technical result.

\begin{lemma}
\label{foias}
 Let
$A:=(A_1,\ldots, A_n)$ and $B:=(B_1,\ldots, B_n)$ be in
$[B(\cH)^n]_1^-$ such that $A\overset{H}{\sim}\, B$, and let
$\{f_k\}_{k=1}^\infty$ be a sequence of elements in $\cA_n\otimes
M_{m}$, $m\in \NN$,  such that $\|f_k\|\leq 1$ for any $k\in \NN$.
Then
$$\lim_{k\to \infty}\|f_k(A_1,\ldots, A_n)\|= 1 \quad \text{ if and
only if }\quad \lim_{k\to \infty}\|f_k(B_1,\ldots, B_n)\|= 1.
$$
\end{lemma}
\begin{proof}
Assume that $\lim_{k\to \infty}\|f_k(A_1,\ldots, A_n)\|= 1$. Then
there is a sequence of vectors $\{h_k\}_{k=1}^\infty\subset
\cH\otimes \CC^m$ such that
$$
\|h_k\|=1 \quad \text{ and }\quad \lim_{k\to
\infty}\|f_k(A_1,\ldots, A_n)h_k\|= 1.
$$
Let $\{\alpha_k\}_{k=1}^\infty \subset \DD$ be such that
$|\alpha_k|\to 1$ as $k\to \infty$. According to Theorem 1.5 from
\cite{Po-free-hol-interp}, the inverse noncommutative Cayley
transform
 $\Gamma^{-1}$ of $\alpha_k
f_r$  is in $\cA_n\otimes M_{m}$ and $\text{\rm
Re\,}\Gamma^{-1}[\alpha_k f_k]\geq 0$. Therefore,
\begin{equation}
\label{Cayley} g_k:=\Gamma^{-1}[\alpha_k f_k]:=(I-\alpha_k f_k)^{-1}
(I+\alpha_k f_k)\in \cA_n\otimes M_{m}
\end{equation}
and $\text{\rm Re\,}g_k\geq 0$ for all $k\in \NN$. Since
$A\overset{H}{\sim}\, B$, Theorem \ref{equivalent} implies the
existence of a constant $c\geq 1$ such that
\begin{equation}
\label{<<2}
 \frac{1}{c^2}\text{\rm Re}\,g_k(A_1,\ldots, A_n)\leq
\text{\rm Re}\,g_k(B_1,\ldots, B_n)\leq c^2 \text{\rm
Re}\,g_k(A_1,\ldots, A_n).
\end{equation}
For each $k\in \NN$, we define the vectors
\begin{equation}
\label{vect} y_k:=[I-\alpha_kf_k(A_1,\ldots, A_n)]h_k \quad \text{
and } \quad  x_k:=[I-\alpha_kf_k(B_1,\ldots, B_n)]y_k.
\end{equation}
Note that due to relations \eqref{Cayley} and \eqref{vect}, we have
\begin{equation*}
\begin{split}
\left<g_k(B_1,\ldots, B_n) y_k, y_k\right>&= \left< \left[
I+\alpha_k f_k(B_1,\ldots, B_n)\right]x_k, \left[ I-\alpha_k
f_k(B_1,\ldots, B_n)\right]x_k\right>\\
&= \|x_k\|^2-|\alpha_k|^2\|f_k(B_1,\ldots, B_n)x_k\|^2 +i\text{\rm
Im\,} \left< \alpha_k f_k(B_1,\ldots, B_n)x_k, x_k\right>.
\end{split}
\end{equation*}
Consequently, we deduce that
$$
\text{\rm Re}\,\left<g_k(B_1,\ldots, B_n) y_k,
y_k\right>=\|x_k\|^2-|\alpha_k|^2\|f_k(B_1,\ldots, B_n)x_k\|^2.
$$
Similarly, we obtain
$$
\text{\rm Re}\,\left<g_k(A_1,\ldots, A_n) y_k,
y_k\right>=\|h_k\|^2-|\alpha_k|^2\|f_k(A_1,\ldots, A_n)h_k\|^2.
$$
Hence and using the second inequality in \eqref{<<2}, we deduce that
$$
\|x_k\|^2-|\alpha_k|^2\|f_k(B_1,\ldots, B_n)x_k\|^2 \leq c^2\left(
\|h_k\|^2-|\alpha_k|^2\|f_k(A_1,\ldots, A_n)h_k\|^2\right).
$$
Consequently, since $\|h_k\|=1$, $ \lim_{k\to
\infty}\|f_k(A_1,\ldots, A_n)h_k\|= 1$ and $|\alpha_k|\to 1$ as
$k\to \infty$, we deduce that
\begin{equation}
\label{conv} \|x_k\|^2-|\alpha_k|^2\|f_k(B_1,\ldots, B_n)x_k\|^2 \to
0,\quad \text{ as } k\to \infty.
\end{equation}
Now, suppose that $\lim_{k\to \infty}\|f_k(B_1,\ldots, B_n)\|\neq
1$. Due to the noncommutative von Neumann inequality we have
$\|f_k(B_1,\ldots, B_n)\|\leq \|f_k\|\leq 1$. Passing to a
subsequence, we can assume that there is $t\in (0,1)$ such that
$\|f_k(B_1,\ldots, B_n)\|\leq t<1$ for all $k\in \NN$. Due to
relation  \eqref{conv} and the fact that
$$
0\leq (1-t^2)\|x_k\|^2\leq \|x_k\|^2-|\alpha_k|^2\|f_k(B_1,\ldots,
B_n)x_k\|^2,
$$
we deduce that  $\|x_k\|\to 0$ as $k\to \infty$. Using relation
\eqref{vect}, we obtain  $y_k=[I-\alpha_kf_k(B_1,\ldots, B_n)]^{-1}
x_k$. Consequently, and using again relation  \eqref{conv},  we have
 \begin{equation}
 \label{fo}
\|[I-\alpha_k f_k(A_1,\ldots, A_n)]h_k\|= \|y_k\|\to 0,\quad \text{
as } \ k\to
 \infty,
 \end{equation}
for any sequence $\{\alpha_k\}_{k=1}^\infty \subset \DD$ with the
property that  $|\alpha_k|\to 1$ as $k\to \infty$. Let
$\{\beta_k\}_{k=1}^\infty \subset \DD$ be another sequence with the
same property and such that $\alpha_k+\beta_k\to 0$ as $k\to
\infty$. Then, due to \eqref{fo}, we have
$$
\lim_{k\to \infty}\|2 h_k-(\alpha_k+ \beta_k)f_k(A_1,\ldots, A_n)
h_k\|=0.
$$
Taking into account that $\|f_k(A_1,\ldots, A_n) h_k\|\leq 1$ for
all $k\in \NN$, we deduce that $\|h_k\|\to 0$ as $k\to \infty$,
which contradicts the fact that $\|h_k\|=1$ for all $k\in \NN$.
Therefore, we must have $\lim_{k\to \infty}\|f_k(B_1,\ldots,
B_n)\|=1$. The converse follows in a similar manner if one uses the
first inequality in \eqref{<<2}. The proof is complete.
\end{proof}

 Now, we have all the ingredients to  we determine the Harnack part of $0$ and obtain  a
  characterization in terms of the free pluriharmonic  Poisson kernel.
\begin{theorem}
\label{foias2} Let $A:=(A_1,\ldots, A_n)$   be in $[B(\cH)^n]_1^-$.
Then the following statements are equivalent:
\begin{enumerate}
\item[(i)]
$A\overset{H}{\sim}\, 0$;
\item[(ii)] $A\in [B(\cH)^n]_1$;
\item[(iii)]
 $r(A_1,\ldots, A_n)<1$ and $P(A, R)\geq aI$ for some constant $a>0$,
 where $P(A, R)$ is
the free pluriharmonic  Poisson kernel at $A$.
\end{enumerate}
\end{theorem}

\begin{proof}
First, we prove the implication $(i)\implies (ii)$. Assume that
$A\overset{H}{\sim}\, 0$ and $\|A\|=1$.  For each $k\in \NN$ define
$$
f_k:=\left[\begin{matrix} S_1&\cdots& S_n\\
0&\cdots &0\\
\vdots&  & \vdots\\
0&\cdots & 0\end{matrix}\right]\in \cA_n\otimes M_{n}.
$$
Then $\|f_k(A_1,\ldots, A_n)\|=\|A\|=1$. Applying Lemma \ref{foias},
we deduce that $0=\|f_k(0)\|\to 1$, as $k\to \infty$, which is a
contradiction. Therefore $\|A\|<1$.

Now, we prove that $(ii)\implies (i)$.  Let $A:=(A_1,\ldots, A_n)$
be in $[B(\cH)^n]_1$ and let $r:=\|A\|<1$. According to Theorem
\ref{Harnack}, we have
$$
  u(0) \,\frac{1-r}{1+r}\leq u(X_1,\ldots, X_n) \leq  u(0)
   \,\frac{1+r}{1-r}
$$
 for any   positive free pluriharmonic
function  $u$ on $[B(\cH)^n]_1$  with operator-valued coefficients
in $B(\cE)$. Applying now Theorem \ref{equivalent2}, we deduce that
$A\overset{H}{\sim}\, 0$.

To prove that $(ii)\implies (iii)$, let $A\in [B(\cH)^n]_1$.  Note
that $r(A_1,\ldots, A_n)\leq \|A\|<1$, which implies the operators
$I-R_A$ and $I-A_1A_1^*-\cdots -A_nA_n^*\geq 0$ are invertible.
Hence, and using the fact that
$$
P(A, R)=(I-R_A^*)^{-1}\left[  (I-A_1A_1^*-\cdots -A_nA_n^*)\otimes I
\right](I-R_A)^{-1},
$$
one can easily deduce that there exists $a>0$ such that $P(A,R)\geq
aI$. Therefore (iii) holds. It remains to show that $(iii)\implies
(i)$. Assume that (iii) holds.  Due to Theorem \ref{A<0}, we have
$A\overset{H}{{ \prec}}\, 0$ and
$$P(X,S)=\sum_{k=1}^\infty\sum_{|\alpha|=k}
 X_\alpha^*\otimes S_{\widetilde\alpha}
+I+\sum_{k=1}^\infty\sum_{|\alpha|=k} X_\alpha\otimes
S_{\widetilde\alpha}^*,\qquad X\in [B(\cH)^n]_1,
$$
where the series are convergent in the operator norm topology. On
the other hand, since $P(A,S)\geq aI$, applying the noncommutative
Poisson transform $\text{\rm id}\otimes P_{rR}$, we obtain
$$
P(rA,R)=(\text{\rm id}\otimes P_{rR})[P(A,S)]\geq aI=aP(0,R)
$$
for any $r\in [0,1)$. Due to   Theorem \ref{equivalent}, equivalence
$(i)\leftrightarrow (iii)$, we have
 $0\overset{H}{{ \prec}}\, A$.  Therefore, item (i)
holds. The proof is complete.
\end{proof}

We remark that, when $n=1$,  we recover a  result
  obtained by Foia\c s \cite{Fo}.

\bigskip

\section{ Hyperbolic metric  on the Harnack parts of the closed ball  $[B(\cH)^n]_1^-$}

In this section we  introduce  a hyperbolic   ({\it
Poincar\'e-Bergman} type) metric $\delta$
  on the Harnack parts of $[B(\cH)^n]_1^-$, and  prove that it is  invariant under
the action of the group $Aut([B(\cH)^n]_1)$  of all the free
holomorphic automorphisms of the noncommutative ball $[B(\cH)^n]_1$.
We   obtain  an explicit formula for the hyperbolic distance in
terms of the reconstruction operator and show that
$\delta|_{\BB_n\times \BB_n}$ coincides with the Poincar\'e-Bergman
distance on $\BB_n$, the open unit ball of $\CC^n$.

Given $A,B\in [B(\cH)^n]_1^-$ in the same Harnack part, i.e.,
$A\,\overset{H}{\sim}\, B$, we introduce
\begin{equation}
\label{om} \omega(A,B):=\inf\left\{ c > 1: \
A\,\overset{H}{{\underset{c}\sim}}\, B   \right\}.
\end{equation}

\begin{lemma}\label{omega}
Let $\Delta$ be a Harnack part of $[B(\cH)^n]_1^-$ and let $A,B,C\in
\Delta$. Then the following properties hold:
\begin{enumerate}
\item[(i)] $\omega(A,B)\geq 1$;
\item[(ii)] $\omega(A,B)=1$ if and only if $A=B$;
\item[(iii)] $\omega(A,B)=\omega(B,A)$;
\item[(iv)] $\omega(A,C)\leq \omega(A,B) \omega(B,C)$.
\end{enumerate}
\end{lemma}
\begin{proof} The items (i) and (iii) are obvious due to relations
\eqref{<<} and \eqref{om}. If $\omega(A,B)=1$, then
$A\,\overset{H}{{\underset{1}\sim}}\, B$ and, due to Theorem
\ref{equivalent}, part (iii), we have $P(rA, R)=P(rB, R)$ for any
$r\in [0,1)$.  Applying this equality to vectors of the form
$h\otimes 1$, $h\in \cH$, we obtain
$$
\sum_{|\alpha|\geq 1} r^{|\alpha|} A_\alpha ^*h\otimes
 e_\alpha=\sum_{|\alpha|\geq 1}   r^{|\alpha|} B_\alpha ^*h\otimes  e_\alpha.
$$
Hence, we deduce that $A_i=B_i$ for any $i=1,\ldots, n$. Therefore,
(ii) holds.

 Note that, due to Theorem \ref{equivalent2},   we have

$$
\frac{1}{\omega(A,B)^2} P(rB,R)\leq P(rA,R)\leq \omega(A,B)^2 P(
rB,R)$$ and
$$
\frac{1}{\omega(B,C)^2} P(rC,R)\leq P(rB,R)\leq \omega(B,C)^2 P(
rC,R)$$
 for any $r\in [0,1)$.
 Consequently, we deduce that
 $$
\frac{1}{\omega(A,B)^2\omega(B,C)^2} P(rC,R)\leq P(rA,R)\leq
\omega(A,B)^2 \omega(B,C)^2 P(rC,R)$$ for any $r\in [0,1)$. Applying
again Theorem \ref{equivalent2}, we have $\omega(A,C)\leq
\omega(A,B) \omega(B,C)$. This completes the proof.
\end{proof}

Now, we can introduce  a  hyperbolic ({\it Poincar\'e-Bergman} type)
metric
  on the Harnack parts of $[B(\cH)^n]_1^-$.

\begin{proposition}\label{delta}
Let $\Delta$ be a Harnack part of $[B(\cH)^n]_1^-$ and define
$\delta:\Delta\times \Delta \to \RR^+$ by setting
\begin{equation}
\label{hyperbolic}
 \delta(A,B):=\ln \omega(A,B),\quad A,B\in \Delta.
\end{equation}
Then $\delta$ is a metric on  $\Delta$.
\end{proposition}
\begin{proof} The result follows from  Lemma \ref{omega}.
\end{proof}

In \cite{Po-automorphism}, we showed that any free holomorphic
automorphism $\Psi$ of the unit ball $[B(\cH)^n]_1$ which fixes the
origin is implemented by a unitary operator on $\CC^n$, i.e., there
is a unitary operator $U$ on $\CC^n$  such that
\begin{equation*}
 \Psi(X_1,\ldots X_n)= \Psi_U(X_1,\ldots
X_n):=[X_1\cdots  X_n]U , \qquad (X_1,\ldots X_n)\in  [B(\cH)^n]_1.
\end{equation*}
The  theory of noncommutative characteristic functions for row
contractions (see \cite{Po-charact}, \cite{Po-varieties}) was  used
to find all the involutive free holomorphic automorphisms of
$[B(\cH)^n]_1$. They  turned out to be of the form
\begin{equation*}
 \Psi_\lambda=- \Theta_\lambda(X_1,\ldots, X_n):={
\lambda}-\Delta_{ \lambda}\left(I_\cK-\sum_{i=1}^n \bar{{
\lambda}}_i X_i\right)^{-1} [X_1\cdots X_n] \Delta_{{\lambda}^*},
\end{equation*}
for some $\lambda=(\lambda_1,\ldots, \lambda_n)\in \BB_n$, where
$\Theta_\lambda$ is the characteristic function  of the row
contraction $\lambda$, and $\Delta_{ \lambda}$,
$\Delta_{{\lambda}^*}$ are certain defect operators.  Moreover,  we
determined the group $Aut([B(\cH)^n]_1)$  of all the free
holomorphic automorphisms of the noncommutative ball $[B(\cH)^n]_1$
and showed that if $\Psi\in Aut([B(\cH)^n]_1)$ and
$\lambda:=\Psi^{-1}(0)$, then there is a unitary operator $U$ on
$\CC^n$ such that
$$
\Psi=\Psi_U\circ \Psi_\lambda.
$$

The following result is essential for the proof of the main result
of this section.

\begin{lemma}
\label{auto} Let   $A:=(A_1,\ldots, A_n)$ and $B:=(B_1,\ldots, B_n)$
be   in $[B(\cH)^n]_1^-$ and let $\Psi\in Aut([B(\cH)^n]_1)$. Then
$A\overset{H}{{\underset{c}\prec}}\, B$ if and only if
$\Psi(A_1,\ldots, A_n)\overset{H}{{\underset{c}\prec}}\,
\Psi(B_1,\ldots B_n)$.
\end{lemma}
\begin{proof}
Let
$$
p(S_1,\ldots, S_n):=\sum_{|\alpha|\leq k} S_\alpha \otimes
M_{(\alpha)},\qquad M_{(\alpha)}\in M_{m},
$$
be a polynomial in $\cA_n\otimes M_{m}$  such that $\text{\rm
Re}\,p(S_1,\ldots, S_n)\geq 0$. Due to the results from
\cite{Po-automorphism}, if $\Psi\in Aut([B(\cH)^n]_1)$, then
$\Psi(S_1,\ldots, S_n)=(\Psi_1(S_1,\ldots, S_n),\ldots,
\Psi_n(S_1,\ldots,S_n))$ is a row contraction with entries
$\Psi_i(S_1,\ldots, S_n)$, $i=1,\ldots, n$, in the noncommutative
disc algebra $\cA_n$. Since the noncommutative Poisson transform at
$\Psi(S_1,\ldots, S_n)$ is a c.p. linear map $P_{\Psi(S_1,\ldots,
S_n)}:C^*(S_1,\ldots, S_n)\to B(\cH)$, we deduce that
\begin{equation*}
\begin{split}
Q(S_1,\ldots, S_n)&:=\text{\rm Re}\,\left[\sum_{|\alpha|\leq k}
\Psi_\alpha(S_1,\ldots, S_n) \otimes M_{(\alpha)}\right]\\
&=(P_{\Psi(S_1,\ldots, S_n)}\otimes \text{\rm id})\left[\text{\rm
Re}\,p(S_1,\ldots, S_n)\right] \geq 0.
\end{split}
\end{equation*}
If  we assume that $A\overset{H}{{\underset{c}\prec}}\, B$, then
Theorem \ref{equivalent} part (v)  implies
$$
Q(A_1,\ldots, A_n)=(P_A\otimes \text{\rm id}))[Q(S_1,\ldots, S_n)]
\leq c^2 (P_B\otimes \text{\rm id}))[Q(S_1,\ldots, S_n)]
 = c^2
Q(B_1,\ldots, B_n),
$$
where $P_A$ and $P_B$ are the noncommutative Poisson transforms at
$A$ and $B$, respectively. Therefore, we have
$$
\text{\rm Re}\, p(\Psi_1(A_1,\ldots, A_n),\ldots, \Psi_n(A_1,\ldots,
A_n))\leq c^2 \text{\rm Re}\, p(\Psi_1(B_1,\ldots, B_n),\ldots,
\Psi_n(B_1,\ldots, B_n)),
$$
which shows that $\Psi(A_1,\ldots,
A_n)\overset{H}{{\underset{c}\prec}}\, \Psi(B_1,\ldots B_n)$.

Conversely, assume that  $\Psi(A_1,\ldots,
A_n)\overset{H}{{\underset{c}\prec}}\, \Psi(B_1,\ldots B_n)$.
Applying the first part of the proof, we deduce that
$$
\Psi^{-1}[\Psi(A_1,\ldots, A_n)]\overset{H}{{\underset{c}\prec}}\,
\Psi^{-1}[\Psi(B_1,\ldots B_n)].
$$
Since $\Psi^{-1}\circ \Psi=\text{\rm id}$ on the closed ball
$[B(\cH)^n]_1^-$, we deduce that
$A\overset{H}{{\underset{c}\prec}}\, B$. The proof is complete.
\end{proof}

Here is the main result of this section.

\pagebreak

\begin{theorem}\label{P-B} Let $\delta :[B(\cH)^n]_1\otimes
[B(\cH)^n]_1\to [0,\infty)$ be the hyperbolic metric. Then the
following statements hold.
\begin{enumerate}
\item[(i)] If    $A, B\in [B(\cH)^n]_1$, then
\begin{equation*}
\delta(A,B)=\ln \max \left\{ \left\|C_{A} C_{B}^{-1} \right\|,
  \left\|C_{B} C_{A}^{-1} \right\|\right\},
\end{equation*}
where $C_X:=( \Delta_X\otimes I)(I-R_X)^{-1}$ and $R_X:=X_1^*\otimes
R_1+\cdots + X_n^*\otimes R_n$ is the reconstruction operator
associated with the right creation operators $R_1,\ldots, R_n$ and
$X:=(X_1,\ldots, X_n)\in [B(\cH)^n]_1$.

\item[(ii)]  For any free
holomorphic automorphism $\Psi$ of the  noncommutative unit ball
$[B(\cH)^n]_1$,
$$\delta(A,B)=\delta(\Psi(A), \Psi(B)),\qquad A,B\in [B(\cH)^n]_1.
$$
\item[(iii)]
$\delta|_{\BB_n\times \BB_n}$ coincides with the Poincar\'e-Bergman
distance on $\BB_n$, i.e.,
$$
\delta(z,w)=\frac{1}{2}\ln
\frac{1+\|\psi_z(w)\|_2}{1-\|\psi_z(w)\|_2},\qquad z,w\in \BB_n,
$$
where $\psi_z$ is the involutive automorphism of $\BB_n$ that
interchanges $0$ and $z$.
\end{enumerate}
\end{theorem}

\begin{proof} Let $A,B\in [B(\cH)^n]_1$. Due to Theorem
\ref{foias2}, we have $A\overset{H}{\sim}\, B$.  In order to
determine $\omega(A,B)$,  assume that
$A\overset{H}{{\underset{c}\sim}}\, B$ for some $c\geq 1$. According
to Theorem \ref{equivalent2}, we have
$$
\frac{1}{c^2} P(rB,R)\leq P( rA,R)\leq c^2 P(rB,R)$$ for any $r\in
[0,1)$. Since $\|A\|<1$ and $\|A\|<1$, we can take the limit, as
$r\to 1$, in the operator norm topology, and obtain
\begin{equation}
\label{PPP}
 \frac{1}{c^2} P(B,R)\leq P(A,R)\leq c^2 P(B,R).
 \end{equation}
  We
recall that the free pluriharmonic  kernel $P(X,R)$,
$X:=(X_1,\ldots, X_n)\in [B(\cH)^n]_1$, has the factorization $P(X,
R)=C_X^* C_X$, where $C_X:=(  \Delta_X\otimes I)(I-R_X)^{-1}$.  Note
also that $C_X$ is an invertible operator. It is easy to see that
relation \eqref{PPP} implies
$$
{C_A^*}^{-1} C_B^* C_B C_A^{-1}\leq c^2 I\quad \text{ and } \quad
{C_B^*}^{-1} C_A^*C_AC_B^{-1}\leq c^2I.
$$
Therefore,
$$
d:=\max \left\{ \left\|C_{A} C_{B}^{-1} \right\|,
  \left\|C_{B} C_{A}^{-1} \right\|\right\}\leq c,
  $$
which implies $d\leq \omega(A,B)$. On the other hand, since
$\left\|C_{B} C_{A}^{-1} \right\|\leq d$ and $ \left\|C_{A}
C_{B}^{-1} \right\|\leq d$, we have

$$
{C_A^*}^{-1} C_B^* C_B C_A^{-1}\leq d^2 I\quad \text{ and } \quad
{C_B^*}^{-1} C_A^*C_AC_B^{-1}\leq d^2I.
$$
Hence, we deduce that
$$
\frac{1}{d^2} C_B^*C_B\leq C_A^*C_A\leq d^2C_B^*C_B,
$$
which is equivalent to
$$\frac{1}{d^2} P(B,S)\leq P(A,S)\leq d^2 P(B,S),
$$
where $S:=(S_1,\ldots, S_n)$ is  the $n$-tuple of left creation
operators.  Applying the noncommutative Poisson transform $\text{\rm
id}\otimes P_{rR }$, $r\in [0,1)$, and taking into account that it
is a positive map, we deduce that
$$\frac{1}{d^2} P(rB,R)\leq P(rA,R)\leq d^2 P(rB,R)
$$
for any $r\in [0,1)$. Due to Theorem \ref{equivalent2}, we deduce
that $A\overset{H}{{\underset{d}\sim}}\, B$ and, consequently,
$\omega(A,B)\leq d$. Since the reverse inequality was already
proved, we have $\omega(A,B)=d$, which together with
\eqref{hyperbolic}  prove part (i).

To prove (ii), let  $\Psi\in Aut([B(\cH)^n]_1)$.  If $A,B\in
[B(\cH)^n]_1$, then, due to Theorem \ref{foias2}, we have
$A\overset{H}{\sim}\, B$. Applying Lemma \ref{auto}, the result
follows.

Now, let us prove item (iii). Let $z:=(z_1,\ldots, z_n)\in \BB_n$.
Due to part (i) of this theorem, we have
$$
\delta(z, 0)=\ln \max \left\{ \|C_z\|, \|C_z^{-1}\|\right\},
$$
where  $C_z:=(1-\|z\|_2)^{1/2}\left(I-\sum_{i=1}^n \bar z_i
R_i\right)^{-1}$. First, we show that
\begin{equation}
\label{eq1} \left\|I-\sum_{i=1}^n \bar z_i R_i\right\|=1+\|z\|_2.
\end{equation}
Indeed, since $R_1,\ldots, R_n$ are isometries with orthogonal
ranges, we have
$$
\left\|\sum_{i=1}^n \bar z_i R_i\right\|=\left(\sum_{i=1}^n
|z_i|^2\right)^{1/2}=\|z\|_2.
$$
Consequently,
\begin{equation}
\label{ine1}
 \left\|I-\sum_{i=1}^n \bar z_i R_i\right\|\leq 1+\|z\|_2.
\end{equation}
Note that, due to Riesz representation theorem, we have
\begin{equation}
\label{equ2} \sup_{w=(w_1,\ldots, w_n)\in \overline{\BB}_n}
\left|1+\sum_{i=1} \bar z_iw_i\right|=1+\|z\|_2.
\end{equation}
On the other hand, due to the noncommutative von Neumann inequality
\cite{Po-von}, we have
\begin {equation}
\label{ine2} \left|1+\sum_{i=1} \bar z_iw_i\right|\leq
\left\|I-\sum_{i=1}^n \bar z_i R_i\right\|
\end{equation}
for any $(w_1,\ldots, w_n)\in \overline{\BB}_n$. Combining relations
\eqref{ine1}, \eqref{equ2}, and \eqref{ine2}, we deduce \eqref{eq1}.
Consequently, we  have
\begin{equation}
\label{C-1} \left\| C_z^{-1}\right\|=\left(\frac{1+\|z\|_2
}{1-\|z\|_2 }\right)^{1/2}.
\end{equation}
Note also that
\begin{equation*}
\begin{split}
\left\|\left(I-\sum_{i=1}^n \bar z_i R_i\right)^{-1}\right\|&\leq 1+
\left\|\sum_{i=1}^n \bar z_i R_i\right\|+\left\|\sum_{i=1}^n \bar
z_i R_i\right\|^2+\cdots\\
&=1+\|z\|_2+ \|z\|_2^2+\cdots\\
&= \frac{1}{1-\|z\|_2}.
\end{split}
\end{equation*}
Consequently, we have
\begin{equation}
\label{Cz} \left\| C_z\right\|\leq \left(\frac{1+\|z\|_2 }{1-\|z\|_2
}\right)^{1/2}.
\end{equation}

Due to relations \eqref{C-1} and \eqref{Cz}, we have
\begin{equation}
\label{z,0} \omega(z,0)=\left(\frac{1+\|z\|_2 }{1-\|z\|_2
}\right)^{1/2}.
\end{equation}
Now, we consider the general case. For each $w\in \BB_n$,
 let $\Psi_z$ be the corresponding involutive automorphism of
$[B(\cH)^n]_1$. We recall (see \cite{Po-automorphism}) that
$\Psi_w(0)=w$ and $\Psi_w(w)=0$. Due to part (ii) of this theorem
and relation \eqref{z,0}, we have
\begin{equation*}
\begin{split}
\delta(z,w)&= \delta(\Psi_w(z), \Psi_w(w))\\
&=\delta(\Psi_w(z), 0)=\ln \omega(\Psi_w(z),0)\\
&=\frac{1}{2}\ln \frac{1+\|\Psi_z(w)\|_2}{1-\|\Psi_z(w)\|_2}.
\end{split}
\end{equation*}
Since, according  to \cite{Po-automorphism}, $\Psi_w$ is a
noncommutative extension of the  involutive automorphism of $\BB_n$
that interchanges $0$ and $z$, i.e.,   $\Psi_w(z)=\psi_w(z)$ for
$z\in \BB_n$, item (iii) follows. The proof is  complete.
\end{proof}

\begin{corollary} \label{delta-ine} For any $X,Y\in [B(\cH)^n]_1$,
$$\delta(X,Y)\leq   \ln
\frac{(\|\Delta_X\|\|\Delta_Y\|}{(1-\|X\|)(1-\|Y\|)}.
$$
\end{corollary}
\begin{proof}
According  to Theorem \ref{foias2} and  Proposition \ref{delta},
$[B(\cH)^n]_1$ is  the Harnack part of $0$ and  $\delta$ is a metric
on the open ball $[B(\cH)^n]_1$. Therefore $\delta(X,Y)\leq
\delta(X,0)+\delta(0,Y)$.
 Theorem \ref{P-B} part (i) implies
$$
\delta(X,0)=\ln\max\{\|C_X\|, \|C_X^{-1}\|\},
$$
where $C_X:= (I\otimes \Delta_X)(I-R_X)^{-1}$ and $R_X:=X_1^*\otimes
R_1+\cdots + X_n^*\otimes R_n$. Since $R_1,\ldots, R_n$ are
isometries with orthogonal ranges, we have
\begin{equation*}
\|I-R_X\|\leq 1+\|R_X\|=1+\|X\|
\end{equation*}
and
\begin{equation*}
\begin{split}
\left\| (I-R_X)^{-1}\right\|&\leq 1+\|R_X\|+\|R_X\|^2+\cdots\\
&=1+\|X\|+\|X\|^2+\cdots\\
&=\frac{1}{1-\|X\|}.
\end{split}
\end{equation*}
On the other hand, since $\|X\|<1$, we have
\begin{equation*}
\begin{split}
\|\Delta_X^{-1}\|^2 &\leq 1+\|XX^*\|+\|XX^*\|^2+\cdots\\
&=\frac{1}{1-\|X\|^2}.
\end{split}
\end{equation*}
Now, one can easily see that
\begin{equation*}
\begin{split}
\|C_X^{-1}\|&=\|(I-R_X)(I\otimes \Delta_X^{-1})\|\\
&\leq (1+\|X\|)\|\Delta_X^{-1}\|\\
&\leq \left(\frac{1+\|X\|}{1-\|X\|}\right)^{1/2}
\end{split}
\end{equation*}
and
$$
\|C_X\|\leq \frac{\|\Delta_X\|}{1-\|X\|}.
$$
Note also that,  due to the fact that $\|I-XX^*\|\geq 1-\|XX^*\|$,
we have
$$
\frac{\|\Delta_X\|}{1-\|X\|}\geq
\left(\frac{1+\|X\|}{1-\|X\|}\right)^{1/2}.
$$
Therefore
$$
\delta(X,0)=\ln \max\{\|C_X\|, \|C_X^{-1}\|\}\leq \ln
\frac{\|\Delta_X\|}{1-\|X\|}.
$$
Taking into account that  $\delta(X,Y)\leq \delta(X,0)+\delta(0,Y)$,
the result follows. The proof is complete.
\end{proof}

In what follows we prove that the    hyperbolic  metric $\delta$, on
the Harnack parts of  $[B(\cH)^n]_1^-$, is invariant under the
automorphism group $Aut([B(\cH)^n]_1)$, and can be  written in terms
of the  the reconstruction operator.

First, we need the following result.
\begin{lemma}
\label{OMr}
  Let   $A:=(A_1,\ldots, A_n)$ and $B:=(B_1,\ldots,
B_n)$ be   in $[B(\cH)^n]_1^-$. Then the following properties hold.
\begin{enumerate}
\item[(i)] $A\overset{H}{\sim}\, B$ if and only if $rA\overset{H}{\sim}\,
rB$ for any $r\in [0,1)$ and  \  $\sup_{r\in [0,
1)}\omega(rA,rB)<\infty$. In this case,
$$
\omega(A,B)=\sup_{r\in [0, 1)}\omega(rA,rB)\quad \text{ and } \quad
\delta(A,B)=\sup_{r\in [0, 1)}\delta(rA,rB).
$$
\item[(ii)]  If  $A\overset{H}{\sim}\, B$, then the functions $r\mapsto \omega(rA,rB)$ and $r\mapsto
\delta(rA,rB)$ are increasing on $[0,1)$.
\end{enumerate}
\end{lemma}

\begin{proof}
Assume that $A\overset{H}{\sim}\, B$ and let $a:=\omega(A,B)$. Due
to relation \eqref{<<}, it is clear that
$A\,\overset{H}{{\underset{a}\sim}}\, B$. By Theorem
\ref{equivalent2}, we deduce that
$$
\frac{1}{a^2} P(rB,R)\leq P(rA,R)\leq a^2 P(rB,R)$$ for any $r\in
[0,1)$, which shows that $rA\,\overset{H}{{\underset{a}\sim}}\, rB$
for any $r\in [0,1)$, and $\sup_{r\in [0, 1)}\omega(rA,rB)\leq a.$

Conversely, assume that $\sup_{r\in [0, 1)}\omega(rA,rB)<\infty$.
Denote $c_r:=\omega(rA,rB)$ for $r\in [0,1)$, and let
$c_1:=\sup_{r\in [0,1)} c_r<\infty$. Since
$rA\,\overset{H}{{\underset{c_r}\sim}}\, rB$ and $1\leq c_r\leq
c_1$, we have
$$
\frac{1}{c_1^2} P(trB,R)\leq \frac{1}{c_r^2} P(trB,R)\leq P(
trA,R)\leq c_r^2 P( trB,R)\leq c_1^2 P( trB,R)$$ for any $t,r\in
[0,1)$. Due to Theorem \ref{equivalent2}   we deduce  that
$A\,\overset{H}{{\underset{c_1}\sim}}\, B$. Thus $\omega(A,B)\leq
c_1$. Since the reverse inequality was already proved above, we have
$a=c_1$. The second part of item (i) is now obvious.

To prove (ii), let $s,t\in [0,1)$ be such that  $s<t$. Applying part
(i), we have
$$
\omega(sA,sB)\leq \sup_{r\in[0,1)} \omega(rtA, rtB)\leq
\omega(tA,tB).
$$
Hence, we deduce   item (ii). The proof is complete.
\end{proof}

\begin{theorem}\label{formula}
Let   $A:=(A_1,\ldots, A_n)$ and $B:=(B_1,\ldots, B_n)$ be   in
$[B(\cH)^n]_1^-$ such that $A\overset{H}{\sim}\, B$. Then
\begin{enumerate}
\item[(i)]
$\delta(A,B)=\delta(\Psi(A), \Psi(B)) $  for any $\Psi\in
Aut([B(\cH)^n]_1)$;
\item[(ii)] the metric $\delta$ satisfies the relation
\begin{equation*}
\begin{split}
\delta(A,B)=\ln \max \left\{\sup_{r\in [0,1)}\left\|C_{rA}
C_{rB}^{-1} \right\|, \sup_{r\in [0,1)} \left\|C_{rB} C_{rA}^{-1}
\right\|\right\},
\end{split}
\end{equation*}
 where $C_X:=( \Delta_X\otimes I)(I-R_X)^{-1}$ and
$R_X:=X_1^*\otimes R_1+\cdots + X_n^*\otimes R_n$ is the
reconstruction operator  associated with the right creation
operators $R_1,\ldots, R_n$ and $X:=(X_1,\ldots, X_n)\in
[B(\cH)^n]_1$.
\end{enumerate}
\end{theorem}

\begin{proof} Due to  Lemma \ref{OMr} and Theorem \ref{P-B}, the
result follows.
\end{proof}

\bigskip

\section{Metric topologies on  Harnack parts of $[B(\cH)^n]_1^-$}

In this section we study the relations between the
$\delta$-topology, the $d_H$-topology (which will be introduced),
and the operator norm topology on  Harnack parts of
$[B(\cH)^n]_1^-$. We prove that the hyperbolic metric  $\delta$ is a
complete metric on certain  Harnack parts, and that all the
topologies above coincide on the open ball $[B(\cH)^n]_1$.

First, we need some notation. Denote
$$
[B_0(\cH)^n]_1^-:=\left\{ (X_1,\ldots, X_n)\in [B(\cH)^n]_1^-:\
r(X_1,\ldots, X_n)<1\right\},
$$
where $r(X_1,\ldots, X_n)$ is the joint spectral radius of
$(X_1,\ldots, X_n)$.
 Note that, due to Theorem \ref{A<0}  and  Theorem \ref{foias2}, we have
$$[B(\cH)^n]_1\subset[B_0(\cH)^n]_1^-=\left\{X\in [B(\cH)^n]_1^-:\
X\overset{H}{{ \prec}}\, 0\right\}.
$$
If $A,B$ are in $ [B_0(\cH)^n]_1^-$, then $A\overset{H}{{ \prec}}\,
0$ and $B\overset{H}{{ \prec}}\, 0$. Consequently, there exists
$c\geq 1$ such that, for any $f\in \cA_n\otimes B(\cE)$ with $
\text{\rm Re}\, f\geq 0$, where $\cE$ is a separable infinite
dimensional Hilbert space,  we have
$$
\text{\rm Re}\, f(A)\leq c^2\text{\rm Re}\, f(0)\quad \text{ and }
\quad \text{\rm Re}\, f(B)\leq c^2\text{\rm Re}\, f(0).
$$
Hence, we deduce that $$\|\text{\rm Re}\, f(A)-\text{\rm Re}\,
f(B)\|\leq 2c^2\,\|\text{\rm Re}\, f(0)\|.$$
 Therefore,  it makes sense
to define the map $d_H:[B_0(\cH)^n]_1^-\times [B_0(\cH)^n]_1^-\to
[0,\infty)$ by setting
$$d_H(A,B):=\sup\left\{ \|u(A)-u(B)\|:\  u\in \text{\rm
Re}\,\left(\cA_n\otimes  B(\cE)\right),  \  u(0)=I,\ u\geq
0\right\}.
$$

\begin{proposition}\label{aprox1}
For any $A,B\in [B_0(\cH)^n]_1^-$,
$$d_H(A,B)=\sup \|\text{\rm Re}\, p(A)-\text{\rm Re}\, p(B)\|,
$$
where the supremum is taken over all polynomials $p\in
\CC[X_1,\ldots, X_n]\otimes M_m$, $m\in \NN$, with $p(0)=I$ and
$\text{\rm Re}\, p\geq 0$.
\end{proposition}
\begin{proof}
Let $f\in \cA_n\otimes B(\cE)$ be such that $\text{\rm Re}\,f\geq 0$
and $(\text{\rm Re}\, f)(0)=I$. According to \cite{Po-analytic}, $f$
has a unique  formal Fourier representation
$$
f=\sum_{\alpha\in \FF_n^+} S_\alpha \otimes C_{(\alpha)},\quad
C_{(\alpha)}\in B(\cE).
$$
Moreover, $\lim_{r\to 1} f_r= f$ in the operator norm topology,
where $ f_r=\sum_{k=1}^\infty \sum_{|\alpha|=k} r^{|\alpha|}S_\alpha
\otimes C_{(\alpha)} $ is in $\cA_n\otimes B(\cE)$ and the series is
convergent in the operator norm. Consequently, for any $\epsilon
>0$, there exist $r_\epsilon\in [0,1)$ and $N_\epsilon\in \NN$ such
that
\begin{equation}
\label{PNr} \|p_{r_\epsilon,N_\epsilon}-f\|<\frac{\epsilon}{2},
\end{equation}
where $p_{r_\epsilon,N_\epsilon}:=\sum_{k=0}^{N_\epsilon}
\sum_{|\alpha|=k} r^{|\alpha|}S_\alpha \otimes C_{(\alpha)}$. Define
the polynomial
$q_{\epsilon,r_\epsilon,N_\epsilon}:=\frac{1}{1+\epsilon}\left(
p_{r_\epsilon, N_\epsilon}+\epsilon I\right) $ and note that
$\left(\text{\rm
Re}\,q_{\epsilon,r_\epsilon,N_\epsilon}\right)(0)=I$. On the other
hand, due to \eqref{PNr}, we have $\|\text{\rm
Re}\,p_{r_\epsilon,N_\epsilon}-\text{\rm
Re}\,f\|<\frac{\epsilon}{2}$, which, due to the fact that $\text{\rm
Re}\,f\geq 0$,  implies $\text{\rm
Re}\,q_{\epsilon,r_\epsilon,N_\epsilon}\geq 0$. Now, notice that
$$
\|q_{\epsilon,r_\epsilon,N_\epsilon}-f\|\leq \frac{1}{1+\epsilon}
\|p_{r_\epsilon,N_\epsilon}-f\|+\frac{\epsilon}{1+\epsilon} \|I+f\|,
$$
which together with relation \eqref{PNr} show that $f$ can be
approximated, in the operator norm,  with polynomials
$q=\sum_{k=0}^N S_\alpha \otimes C_{(\alpha)}$, $C_{(\alpha)}\in
B(\cE)$, such that $\text{\rm Re}\, q\geq 0$ and $(\text{\rm Re}\,
q)(0)=I$. Consider now an othonormal basis $\{\xi_1, \xi_2,
\ldots\}$ of $\cE$ and let $\cE_m:=\text{\rm span}\{\xi_1,\ldots,
\xi_m\}$. Setting
$$q_m:=P_{F^2(H_n)\otimes \cE_m} q|_{F^2(H_n)\otimes
\cE_m}=\sum_{k=0}^N S_\alpha \otimes P_{\cE_m}C_{(\alpha)}|_{\cE_m},
$$
it is easy to see that $\text{\rm Re}\, q_m\geq 0$, $(\text{\rm
Re}\,q_m)(0)=I$, and
$$
\|q(A)-q(B)\|=\lim_{m\to\infty} \|q_m(A)-q_m(B)\|.
$$
This can be used to  complete the proof.
\end{proof}

     Due to the next result, we call  $d_H$ the  {\it kernel  metric} on
$[B_0(\cH)^n]_1^- $.

\begin{proposition}
\label{P-P} $d_H$  is a metric  on $[B_0(\cH)^n]_1^-$ satisfying
relation
$$
d_H(A,B)=\|P(A,R)-P(B,R)\|
$$
for any $A,B\in [B_0(\cH)^n]_1^-$, where $P(X,R)$ is the free
pluriharmonic Poisson kernel  associated with $X\in
[B_0(\cH)^n]_1^-$. Moreover,   the map \ $[0,1)\ni r\mapsto
d_H(rA,rB)\in \RR^+$ \ is     increasing   and
$$
d_H(A,B)=\sup_{r\in [0,1)}d_H(rA,rB).
$$
\end{proposition}
\begin{proof} It is easy to see  that $d_H$ is a metric.
Since $r(A)<1$, the map $v:[B(\cK)^n]_1\to B(\cK)\otimes_{min}
B(\cH)$ defined by
$$
v_A(Y):=\sum_{k=1}^\infty \sum_{|\alpha|=k} Y_{\tilde \alpha}\otimes
A_\alpha^*+ I+ \sum_{k=1}^\infty \sum_{|\alpha|=k} Y_{\tilde
\alpha}^*\otimes A_\alpha
$$
is a free pluriharmonic  function on the  open ball
$[B(\cK)^n]_\gamma$ for some $\gamma>1$, where $\cK$ is an infinite
dimensional Hilbert space. Therefore $v_A$ is continuous on
$[B(\cK)^n]_\gamma$ and
$$v_A(S):=\sum_{k=1}^\infty \sum_{|\alpha|=k} S_{\tilde \alpha}\otimes
A_\alpha^*+ I+ \sum_{k=1}^\infty \sum_{|\alpha|=k} S_{\tilde
\alpha}^*\otimes A_\alpha
$$
is in $\text{\rm Re}\,\left(\cA_n\otimes  B(\cH)\right)$.
Consequently, $v_A(S)=\lim_{r\to 1} v_A(rS)$ in the norm topology.
Since a similar result holds for $v_B$,  and $R$ is unitarily
equivalent to $S$, we have
\begin{equation*}
\begin{split}
\|P(A,R)-P(B,R)\|&=\|v_A(S)-v_B(S)\|\\
&=\lim_{r\to 1}\|v_A(rS)-v_B(rS)\|\\
&=\lim_{r\to 1}\|P(rA,R)-P(rB,R)\|.
\end{split}
\end{equation*}
Due to the noncommutative von Neumann inequality the map
$$
[0,1)\ni r\mapsto \|P(rA,R)-P(rB,R)\|\in \RR^+
$$
is increasing.

For each $r\in [0,1)$, the map $u_r(Y):=P(Y, rR)$ is a free
pluriharmonic  on $[B(\cH)^n]_\gamma$ for some $\gamma>1$, with
coefficients in $B(F^2(H_n))$. Moreover,  since $u_r$ is positive
and $u_r(0)=I$,  the definition of $d_H$ implies
\begin{equation}
\label{first-ine-r}\|u_r(A)-u_r(B)\|= \|P(A,rR)-P(B,rR)\|\leq
d_H(A,B)
\end{equation}
for any $r\in [0,1)$.  Taking $r\to 1$ and using the first part of
the proof, we deduce that
\begin{equation}
\label{first-ine} \|P(A,R)-P(B,R)\|\leq d_H(A,B).
\end{equation}

Now, let $G\in \text{\rm Re}\,\left(\cA_n\otimes  B(\cE)\right)$ be
with  $G(0)=I$ and $ G\geq 0$. Since $Y\mapsto G(Y)$ is a positive
free pluriharmonic function of $[B(\cH)^n]_1$, we can apply
Corollary 5.5 from \cite{Po-pluriharmonic} to find a completely
positive linear map $\mu:\cR_n^*+ \cR_n\to B(\cE)$ with $\mu(I)=I$
and
$$
G(Y)=(\text{\rm id}\otimes \mu)[P(Y,R)],\quad Y\in [B(\cH)^n]_1.
$$
Consequently, for each $r\in [0,1)$, we have
$$
\|G(rA)-G(rB)\|\leq \|\mu\|\|P(A,R)-P(B,R)\|.
$$
Since $G\in \text{\rm Re}\,\left(\cA_n\otimes  B(\cE)\right)$,
Theorem 4.1  from \cite{Po-pluriharmonic} shows that
$G(A)=\lim_{r\to 1} G(rA)$ and $G(B)=\lim_{r\to 1} G(rB)$ in the
operator norm topology. Due to the fact that  $\|\mu\|=1$, we have
$$
\|G(A)-G(B)\|\leq \|P(A,R)-P(B,R)\|.
$$
Therefore, $d_H(A,B)\leq \|P(A,R)-P(B,R)\|$, which together with
\eqref{first-ine} prove the equality. The last part of the
proposition can be easily deduced from the considerations above. The
proof is complete.
\end{proof}

\begin{theorem}
\label{ker-metric} Let $d_H$ be the kernel metric on
$[B_0(\cH)^n]_1^-$. Then the following statements hold:
\begin{enumerate}
\item[(i)] the metric $d_H$ is complete on $[B_0(\cH)^n]_1^-$;
\item[(ii)]  the $d_H$-topology is stronger than the norm topology on
$[B_0(\cH)^n]_1^-$;
\item[(iii)]   the $d_H$-topology coincides with the norm topology on  the open unit ball
$[B(\cH)^n]_1$.
\end{enumerate}
\end{theorem}

\begin{proof} First we prove that
\begin{equation}
\label{A-B} \|A-B\|\leq \|P(A,R)-P(B,R)\|, \quad A,B\in
[B_0(\cH)^n]_1^-.
\end{equation}
Indeed, for each $r\in[0,1)$, we have
$$
rR_A=\frac{1}{2\pi} \int_0^{2\pi} e^{it} P(A,re^{it} R) dt,
$$
where $R_A:=A_1^*\otimes R_1+\cdots A_n^*\otimes R_n$ is the
reconstruction operator. Using the noncommutative von Neumann
inequality, we obtain
\begin{equation*}
\begin{split}
\|rA-rB\|&= \|rR_A-rR_B\|\\
&=\left\|\frac{1}{2\pi} \int_0^{2\pi} e^{it} [P(A,re^{it}
R)-P(B,re^{it} R)] dt\right\|\\
&\leq \sup_{t\in[0,2\pi]}\left\|P(A,re^{it} R)-P(B,re^{it}
R)\right\|\\
&\leq \|P(A,rR)-P(B,rR)\|,
\end{split}
\end{equation*}
for any $r\in [0,1)$.
 Since $r(A)<1$ and $r(B)<1$, we have
$$
\lim_{r\to 1} \|P(A,rR)-P(B,rR)\|=\|P(A,R)-P(B,R)\|,
$$
which proves our assertion.

Now, to prove (i), let $\{A^{(k)}:=(A_1^{(k)},\ldots,
A_n^{(k)})\}_{k=1}^\infty$ be a $d_H$-Cauchy sequence  in
$[B_0(\cH)^n]_1^-$. Due to relation \eqref{A-B}, we have
$$
\|A^{(k)}-A^{(p)}\|\leq \|P(A^{(k)},
R)-P(A^{(p)},R)\|=d_H(A^{(k)},A^{(p)})
$$
for any $k,p\in \NN$. Hence, $\{A^{(k)}\}_{k=1}^\infty$ is a Cauchy
sequence in the norm topology of $[B(\cH)^n]_1^-$. Therefore,  there
exists $T:=(T_1,\ldots, T_n)$ in $[B(\cH)^n]_1^-$ such that
$\|T-A^{(k)}\|\to 0$, as $k\to\infty$.

Now let us prove that the joint spectral radius $r(T)<1$. Since
$\{A^{(k)}\}_{k=1}^\infty$  is a $d_H$-Cauchy sequence, there exists
$k_0\in \NN$ such that $d_H(A^{(k)},A^{(k_0)})\leq 1$ for any $k\geq
k_0$. On the other hand, since $A^{(k_0)}\in [B_0(\cH)^n]_1^-$,
Theorem \ref{A<0} shows that $A^{(k_0)} \overset{H}{{ \prec}}\, 0$.
Applying Theorem \ref{equivalent}, we find $c\geq 1$ such that
$P(rA^{(k_0)},R)\leq c^2$ for any $r\in [0,1)$. Hence and using
inequality \eqref{first-ine-r}, we have
\begin{equation*}
\begin{split}
P(rA^{(k)}, R)&\leq \left( \|P(rA^{(k)}, R)-P(rA^{(k_0)}, R)\|
+\|P(rA^{(k_0)}, R)\|\right)I\\
&\leq \left(  d_H(A^{(k)},  A^{(k_0)}) +\|P(rA^{(k_0)},
R)\|\right)I\\
&\leq (1+c^2)I
\end{split}
\end{equation*}
for any $k\geq k_0$ and $r\in [0,1)$. Taking $k\to\infty$ and using
the continuity of the free pluriharmonic functions in the operator
norm topology, we obtain $P(rT, R)\leq (1+c^2)I$ for $r\in [0,1)$.
Applying again Theorem \ref{equivalent}, we deduce that $T
\overset{H}{{ \prec}}\, 0$. Now,  Theorem \ref{A<0} implies
$r(T)<1$, which shows that $T\in [B_0(\cH)^n]_1^-$ and  proves part
(i).

Note that part (ii) follows from Proposition \ref{P-P} and
inequality \eqref{A-B}. To prove part (iii), we assume that $A,B\in
[B(\cH)^n]_1$. First,  recall from the proof of Corollary
\ref{delta-ine} that
\begin{equation*}
\left\| (I-R_X)^{-1}\right\| =\frac{1}{1-\|X\|}
\end{equation*}
for any $X\in [B(\cH)^n]_1$.  Consequently, we have
\begin{equation*}
\begin{split}
\|P(A,R)-P(B,R)\|&\leq 2\|(I-R_A)^{-1}-(I-R_B)^{-1}\|\\
& \leq 2\|(I-R_A)^{-1}\| \|(I-R_B)^{-1}\|\|R_A-R_B\|\\
&\leq \frac{2 \|A-B\|}{(1-\|A\|)(1-\|B\|)}.
\end{split}
\end{equation*}
Hence, using part (ii) and the fact that
$d_H(A,B)=\|P(A,R)-P(B,R)\|$, we  deduce that the $d_H$-topology
coincides with the norm topology on  the open unit ball
$[B(\cH)^n]_1$. This completes the proof.
\end{proof}

\begin{corollary}\label{INE}
If $A,B\in [B_0(\cH)^n]_1^-$, then
$$
\|A-B\|\leq d_H(A,B).
$$
Moreover, if $A,B\in [B(\cH)^n]_1$, then
$$
d_H(A,B)\leq \frac{2 \|A-B\|}{(1-\|A\|)(1-\|B\|)}.
$$
\end{corollary}

In what follows we obtain another formula for the hyperbolic
distance  that will be used to prove the main result of this
section.
  We mention that if $f\in \cA_n\otimes M_m$, $m\in
\NN$, then  we call  $\text{\rm Re}\,f$  strictly positive and
denote $\text{\rm Re}\,f>0$ if there exists a constant $a>0$ such
that $\text{\rm Re}\,f\geq aI$.

\begin{proposition}\label{delta-form}
Let   $A:=(A_1,\ldots, A_n)$ and $B:=(B_1,\ldots, B_n)$ be   in
$[B(\cH)^n]_1^-$ such that $A\overset{H}{\sim}\, B$. Then

\begin{equation}
\label{de-sup}
\delta(A,B)=\frac{1}{2}\sup\left|
\ln\frac{\left<\text{\rm Re}\, f(A_1,\ldots,
A_n)x,x\right>}{\left<\text{\rm Re}\, f(B_1,\ldots,
B_n)x,x\right>}\right|,
\end{equation}
where the supremum is taken over all $f\in \cA_n\otimes M_m$, $m\in
\NN$, with $\text{\rm Re}\,f>0$ and $x\in \cH\otimes \CC^m$ with
$x\neq 0$.

\end{proposition}

\begin{proof}
Denote the right hand side of \eqref{de-sup} by $\delta'(A,B)$. If
 $f\in \cA_n\otimes M_m$, $m\in
\NN$, with $\text{\rm Re}\,f\geq aI$, then applying the
noncommutative Poison transform, we have $\text{\rm
Re}\,f(Y_1,\ldots, Y_n)\geq aI$ for any $(Y_1,\ldots, Y_n)\in
[B(\cH)^n]_1^-$. Assume that $A\overset{H}{\sim}\, B$ with $c\geq
1$. Due to  Theorem \ref{equivalent}, we deduce that
$$
\frac{1}{c^2}\leq \frac{\left<\text{\rm Re}\, f(A_1,\ldots,
A_n)x,x\right>}{\left<\text{\rm Re}\, f(B_1,\ldots,
B_n)x,x\right>}\leq c^2
$$
for any $x\in \cH\otimes \CC^m$ with $x\neq 0$. Hence, we have
$\delta'(A,B)\leq \ln c$, which implies $\delta'(A,B)\leq
\delta(A,B)$.

To prove the reverse inequality, note that if $g\in \cA_n\otimes
M_m$  with $\text{\rm Re}\,g\geq 0$, then $f:=g+\epsilon I$ has the
property that $\text{\rm Re}\,g\geq \epsilon I$ for any
$\epsilon>0$. Consequently,
$$
\frac{1}{2}\sup\left| \ln\frac{\left<\text{\rm Re}\, f(A_1,\ldots,
A_n)x,x\right>}{\left<\text{\rm Re}\, f(B_1,\ldots,
B_n)x,x\right>}\right|\leq \delta'(A,B)
$$
for any $x\in \cH\otimes \CC^m$ with $x\neq 0$. Note that the latter
inequality is equivalent to
\begin{equation*}
\begin{split}
\frac{1}{e^{2\delta'(A,B)}}&\left<\text{\rm Re}\, g(B_1,\ldots,
B_n)x,x\right> +\frac{\epsilon}{e^{2\delta'(A,B)}}\|x\|^2\\
 &\leq \left<\text{\rm Re}\, g(B_1,\ldots,
B_n)x,x\right> +\epsilon \|x\|^2\\
&\leq e^{2\delta'(A,B)}\left<\text{\rm Re}\, g(B_1,\ldots,
B_n)x,x\right> +\epsilon e^{2\delta'(A,B)} \|x\|^2.
\end{split}
\end{equation*}
Taking $\epsilon \to 0$, we deduce that $\omega(A,B)\leq
e^{\delta'(A,B)}$, which implies $\delta(A,B)\leq \delta'(A,B)$ and
completes the proof.
\end{proof}

We remark that, under the conditions of Proposition
\ref{delta-form}, one can  also prove that
\begin{equation*}
 \delta(A,B)=\frac{1}{2}\sup\left| \ln\frac{\left<\text{\rm Re}\,
p(A_1,\ldots, A_n)x,x\right>}{\left<\text{\rm Re}\, p(B_1,\ldots,
B_n)x,x\right>}\right|,
\end{equation*}
 where the supremum is taken over all noncommutative polynomials
  $p\in
\CC[X_1,\ldots, X_n]\otimes M_m$, $m\in \NN$, with $\text{\rm Re}\,
p> 0$, and $x\in \cH\otimes \CC^m$ with $x\neq 0$.

Indeed, if $f\in \cA_n\otimes M_m$ with $\text{\rm Re}\, f\geq aI,$
$a>0$, and $0<\epsilon <a$, then, as in the proof of Proposition
\ref{aprox1}, we can find a polynomial $p=\sum_{|\alpha|\leq N}
S_\alpha \otimes C_{(\alpha)}$, $C_{(\alpha)}\in M_m$, such that
$\|f-p\|<\epsilon$. Hence,  $\|\text{\rm Re}\, f-\text{\rm Re}\,
p\|<\epsilon$ and, consequently, $\text{\rm Re}\, p >0$. Now, our
assertion  follows.

Here is the main result of this section.

\begin{theorem}
\label{topology} Let  $\delta$ be the   Poincar\'e-Bergman type
metric on $[B(\cH)^n]_1^-$ and let $\Delta$ be a Harnack part  of
$[B_0(\cH)^n]_1^-$ . Then  the following properties hold:
 \begin{enumerate}
 \item[(i)]
 $\delta$ is complete on
 $\Delta$;
\item[(ii)]
the $\delta$-topology is stronger then the $d_H$-topology on
$\Delta$;
\item[(iii)]
the $\delta$-topology, the $d_H$-topology, and the operator norm
topology coincide on the open ball $[B(\cH)^n]_1$.
\end{enumerate}
\end{theorem}

\begin{proof} Let   $A:=(A_1,\ldots, A_n)$ and $B:=(B_1,\ldots, B_n)$ be
in a Harnack part $\Delta$ of  $[B_0(\cH)^n]_1^-$.   Then
$A\overset{H}{\sim}\, B$ and \begin{equation*} \text{\rm Re}\,
f(A_1,\ldots, A_n)\leq \omega(A,B)^2\text{\rm Re}\, f(B_1,\ldots,
B_n)
\end{equation*}
 for any $f\in \cA_n\otimes M_m$  with $\text{\rm Re}\,f\geq 0$.
 Hence, we have
 \begin{equation}
 \label{Re-omega}
\text{\rm Re}\, f(A_1,\ldots, A_n)-\text{\rm Re}\, f(B_1,\ldots,
B_n)\leq [\omega(A,B)^2-1]\text{\rm Re}\, f(B_1,\ldots, B_n).
\end{equation}
On the other hand, since $B\overset{H}{{ \prec}}\, 0$, we have
$r(B)<1$ so $P(B,R)$ makes sense. Also, due to the fact that the
noncommutative Poisson transform  $\text{\rm id}\otimes P_{rR}$ is
completely positive, and $P(B,S)\leq \|P(B,R)\| I$, we deduce that
\begin{equation*}
\begin{split}
P(rB,R)&=(\text{\rm id}\otimes P_{rR})[P(B, S)]\leq \|P(B,R)\| I\\
&=\|P(B,R)\| P(0,R)
\end{split}
\end{equation*}
for any $r\in [0,1)$. Applying now the equivalence
$(iii)\leftrightarrow (iv)$  of Theorem \ref{equivalent}, when
$c^2=\|P(B,R)\|$),   we obtain $ \text{\rm Re}\, f(rB_1,\ldots,
rB_n)\leq \|P(B,R)\|\text{\rm Re}\, f(0) $ for any $r\in [0,1)$.
Taking the limit, as  $r\to 1$,  in the operator norm topology, we
get
$$\text{\rm Re}\, f(B_1,\ldots,
B_n)\leq \|P(B,R)\|\text{\rm Re}\, f(0).
$$
Combining this  inequality with \eqref{Re-omega}, we have
$$
\text{\rm Re}\, f(A_1,\ldots, A_n)-\text{\rm Re}\, f(B_1,\ldots,
B_n)\leq [\omega(A,B)^2-1]\|P(B,R)\|\text{\rm Re}\, f(0).
$$
A similar inequality holds if one interchange $A$ with $B$. If, in
addition, we assume that $\text{\rm Re}\, f(0)=I$, then we can
deduce that
$$
-sI\leq \text{\rm Re}\, f(A_1,\ldots, A_n)-\text{\rm Re}\,
f(B_1,\ldots, B_n)\leq sI,
$$
where $s:=[\omega(A,B)^2-1]\max\{\|P(A,R)\|,\|P(B,R)\|\}$. Since $
\text{\rm Re}\, f(A_1,\ldots, A_n)-\text{\rm Re}\, f(B_1,\ldots,
B_n)$ is a self-adjoint operator, we have $\| \text{\rm Re}\,
f(A_1,\ldots, A_n)-\text{\rm Re}\, f(B_1,\ldots, B_n)\|\leq s$. Due
to the definition of the metric $d_H$, we deduce that
 $d_H(A,B)\leq s$. Consequently, we obtain
 \begin{equation}
 \label{ine-dh}
 d_H(A,B)\leq \max\{\|P(A,R)\|,\|P(B,R)\|\}
 \left(e^{2\delta(A,B)}-1\right).
\end{equation}

Now, we prove that $\delta$ is a complete metric on $\Delta$. Let
$\{A^{(k)}:=(A_1^{(k)},\ldots, A_n^{(k)})\}_{k=1}^\infty\subset
\Delta$ be a $\delta$-Cauchy sequence. First, we prove that the
sequence $\{\|P(A^{(k)}, R)\|\}_{k=1}^\infty$ is bounded.  For any
$\epsilon>0$ there exists $k_0\in \NN$ such that
\begin{equation}
\label{de} \delta(A^{(k)}, A^{(p)})<\epsilon\quad \text{ for any } \
k,p\geq k_0.
\end{equation}
Since $A^{(k)}\overset{H}{\sim}\, A^{(k_0)}$ and
$A^{(k_0)}\overset{H}{{ \prec}}\, 0$,  for any   $f\in \cA_n\otimes
M_m$  with $\text{\rm Re}\,f\geq 0$, we have
\begin{equation}
\label{an-ine} \text{\rm Re}\, f(A^{(k)})\leq \omega (A^{(k)},
A^{(k_0)})^2 \text{\rm Re}\, f(A^{(k_0)})\leq c^2 \text{\rm Re}\,
f(0),
\end{equation}
where  $c:=\|P(A^{(k_0)},R)\|^{1/2}\omega (A^{(k)}, A^{(k_0)}) $.
Consequently, due to Theorem \ref{equivalent}, we have
$A^{(k)}\overset{H}{{ \prec}}\, 0$ and  $ \|P(A^{(k)},R)\|\leq c^2$
for any $k\geq k_0$. Combining this with  relation \eqref{de}, we
obtain

$$
\|P(A^{(k)},R)\|\leq \|P(A^{(k_0)},R)\| e^{2\epsilon}
$$
for any $k\geq k_0$. This shows that the sequence $\{\|P(A^{(k)},
R)\|\}_{k=1}^\infty$ is bounded. Consequently, due to inequality
\eqref{ine-dh}, we deduce that $\{A^{(k)}\}$ is a $d_H$-Cauchy
sequence. According to Theorem \ref{ker-metric}, there exists
$A:=(A_1,\ldots, A_n)\in [B_0(\cH)^n]_1^-$ such that
\begin{equation}
\label{conv1} d_H(A^{(k)}, A)\to 0\quad \text{ as } \ k\to \infty.
\end{equation}

Now, let $f\in \cA_n\otimes M_m$  with $\text{\rm Re}\,f\geq 0$ and
$\text{\rm Re}\,f(0)=I$. Due to relations  \eqref{an-ine} and
\eqref{de}, we have
\begin{equation}
\label{Re-ine} \text{\rm Re}\, f(A^{(k)})\leq \omega (A^{(k)},
A^{(k_0)})^2 \text{\rm Re}\, f(A^{(k_0)})\leq
e^{2\epsilon}\text{\rm Re}\, f(A^{(k_0)})
\end{equation}
for $k\geq k_0$. According to \eqref{conv1}, $\text{\rm Re}\,
f(A^{(k)})\to \text{\rm Re}\, f(A)$, as $k\to\infty$,  in the
operator norm topology. Consequently, \eqref{Re-ine} implies
\begin{equation}
\label{ine-Re} \text{\rm Re}\, f(A)\leq   e^{2\epsilon}\text{\rm
Re}\, f(A^{(k_0)}).
\end{equation}

A similar inequality can be deduced in  the more general case when
$f\in \cA_n\otimes M_m$  with $\text{\rm Re}\,f\geq 0$. Indeed, for
each $\epsilon'>0$ let $g:=\epsilon'I+f$, $Y=\text{\rm Re}\, g(0)$,
and $\varphi:=Y^{-1/2} gY^{-1/2}$. Since
 $\text{\rm Re}\,\varphi\geq 0$ and $\text{\rm Re}\,
\varphi(0)=I$, we can apply inequality \eqref{ine-Re} to $\varphi$
and deduce that
$$
\epsilon' I+\text{\rm Re}\, f(A)\leq e^{2\epsilon}\left(\epsilon'
I+\text{\rm Re}\, f(A^{(k_0)})\right)
$$
for any $\epsilon'>0$. Taking $\epsilon'\to 0$, we get

\begin{equation}
\label{ine-Re2} \text{\rm Re}\, f(A)\leq   e^{2\epsilon}\text{\rm
Re}\, f(A^{(k_0)})
\end{equation}
for any $f\in \cA_n\otimes M_m$  with $\text{\rm Re}\,f\geq 0$. This
shows that
\begin{equation}
\label{<1} A\overset{H}{{ \prec}}\, A^{(k_0)}.
\end{equation}
On the other hand, since $A^{(k_0)}\overset{H}{{ \prec}}\, A^{(k)}$
for any $k\geq k_0$, we deduce that
$$
\text{\rm Re}\, f(A^{(k_0)})\leq \omega (A^{(k_0)}, A^{(k)})^2
\text{\rm Re}\, f(A^{(k)})\leq e^{2\epsilon}\text{\rm Re}\,
f(A^{(k)})
$$
for $k\geq k_0$. According to Theorem \ref{ker-metric}, the
$d_H$-topology is stronger than the norm topology on
$[B_0(\cH)^n]_1^-$. Therefore, relation \eqref{conv1} implies
$A^{(k)}\to A\in[B_0(\cH)^n]_1^-$ in the operator norm topology.
Passing to the limit in the inequality above, we deduce that
\begin{equation}
\label{Re3} \text{\rm Re}\, f(A^{(k_0)})\leq e^{2\epsilon}\text{\rm
Re}\, f(A )
\end{equation}
for any $f\in \cA_n\otimes M_m$  with $\text{\rm Re}\,f\geq 0$.
Consequently, we have $A^{(k_0)}\overset{H}{{ \prec}}\, A$. Hence
and using \eqref{<1}, we obtain $A\overset{H}{{ \sim}}\, A^{(k_0)}$,
which proves that $A\in \Delta$. From the inequalities
\eqref{ine-Re2} and \eqref{Re3}, we have $\omega(A^{(k_0)}, A)\leq
e^{2\epsilon}$.  Hence,  $\delta(A^{(k_0)}, A)<\epsilon$, which
together with \eqref{de} imply $\delta(A^{(k)}, A)<2\epsilon$ for
any $k\geq k_0$. Consequently, $\delta(A^{(k)}, A)\to 0$ as $k\to
\infty$, which proves that $\delta$ is complete on $\Delta$.

Note that we have also proved  part (ii) of this theorem. Now, we
prove part (iii). To this end, assume that $A,B\in [B(\cH)^n]_1$.
Due to the fact that $\|B\|<1$, $P(B,R)$ is a positive invertible
operator. Since $P(B,R)^{-1}\leq \| P(B,R)^{-1}\|$, we have $I\leq
\| P(B,R)^{-1}\| P(B,R)$, which implies $I\leq \| P(B,R)^{-1}\|
P(rB,R)$ for ant $r\in [0,1)$. According to Theorem
\ref{equivalent}, we deduce that $ 0\overset{H}{{ \prec}}\, B$ and
$$
\text{\rm Re}\, f(0)\leq \| P(B,R)^{-1}\|\, \text{\rm Re}\, f(B)
$$
for any $f\in \cA_n\otimes M_m$  with $\text{\rm Re}\,f\geq 0$. If,
in addition, we assume that $\text{\rm Re}\,f(0)=I$, then the latter
inequality implies
\begin{equation*}
\begin{split}
\frac{\left<\text{\rm Re}\, f(A)x,x\right>}{\left<\text{\rm Re}\,
f(B)x,x\right>}-1 &\leq
\frac{\|P(B,R)^{-1}\|}{\|x\|}\left<\left(\text{\rm Re}\,
f(A)-\text{\rm Re}\,
f(B)\right)x,x\right>\\
&\leq \|P(B,R)^{-1}\| d_H(A,B)
\end{split}
\end{equation*}
for any $x\in \cH\otimes \CC^m$, $x\neq 0$. Hence, we deduce that
$$
\ln\frac{\left<\text{\rm Re}\, f(A)x,x\right>}{\left<\text{\rm Re}\,
f(B)x,x\right>}\leq \ln \left(1+\|P(B,R)^{-1}\| d_H(A,B)\right).
$$
One can obtain a  similar inequality  interchanging $A$ with $B$.
Combining these two inequalities, we obtain
\begin{equation}
\label{ln1} \left|\ln\frac{\left<\text{\rm Re}\,
f(A)x,x\right>}{\left<\text{\rm Re}\, f(B)x,x\right>}\right|\leq \ln
\left(1+\max\{\|P(B,R)^{-1}\|, \|P(A,R)^{-1}\|\} d_H(A,B)\right).
\end{equation}
Consider  now the general case when $g\in \cA_n\otimes M_m$  with
$\text{\rm Re}\,g> 0$.  Then  $Y:=\text{\rm Re}\, g(0)$ is a
positive invertible operator on $\cH\otimes \CC^m$  and $f:=Y^{-1/2}
gY^{-1/2}$ has the properties $\text{\rm Re}\,f\geq 0$ and
$\text{\rm Re}\,f(0)=I$. Applying \eqref{ln1} to $f$   when
$x:=Y^{-1/2}y$, $y\in \cH\otimes \CC^m$, and $y\neq 0$, we deduce
that
\begin{equation}
\label{2de} 2\delta(A,B)\leq  \ln \left(1+\max\{\|P(B,R)^{-1}\|,
\|P(A,R)^{-1}\|\} d_H(A,B)\right).
\end{equation}

Now, let $\{A^{(k)}\}_{k=1}^\infty$ be a sequence of  elements in
$[B(\cH)^n]_1$ and let $A\in [B(\cH)^n]_1$ be such that
$d_H(A^{(k)}, A)\to 0$ as $k\to \infty$. Due to Proposition
\ref{P-P}, we deduce that $P(A^{(k)},R)\to P(A,R)$ in the operator
norm topology, as $k\to \infty$. On the other hand,  the operators
$P(A^{(k)},R)$ and $ P(A,R)$ are invertible due to the fact that
$\|A^{(k)}\|<1$ and $\|A\|<1$. Consequently, and using the
well-known fact that the map $Z\mapsto Z^{-1}$ is continuous on the
open set of all invertible operators, we deduce that
$P(A^{(k)},R)^{-1}\to P(A,R)^{-1}$ in the operator norm topology.
Hence, the sequence $\{\|P(A^{(k)},R)^{-1}\|\}_{k=1}^\infty$ is
bounded. Therefore, there exists $M>0$ with
$\|P(A^{(k)},R)^{-1}\|\leq M$ for any $k\in \NN$. Applying now
inequality \eqref{2de}, we deduce that
$$
 2\delta(A^{(k)},A)\leq  \ln \left(1+ M d_H(A^{(k)},A)\right)\quad
 \text{ for any } \ k\in \NN.
 $$
Since $d_H(A^{(k)}, A)\to 0$ as $k\to \infty$,  the latter
inequality implies  that $\delta(A^{(k)},A)\to 0$ as $k\to \infty$.
Therefore the $d_H$-topology on $[B(\cH)^n]_1$ is stronger than the
$\delta$-topology. Due to the first part of this theorem, the two
topologies coincide on $[B(\cH)^n]_1$.  Applying now Theorem
\ref{Harnack}, we complete the proof.
\end{proof}

\begin{corollary}\label{INE2}
Let $\Delta$ be a Harnack part  of $[B_0(\cH)^n]_1^-$ . Then
$$
\delta(A,B)\geq \frac{1}{2} \ln\left( 1+\frac{d_H(A,B)}{
\max\{\|P(A,R)\|,\|P(B,R)\|\}}\right),\qquad A,B\in \Delta.
$$
Moreover, if $A,B\in \Delta:=[B(\cH)^n]_1$, then
$$
\delta(A,B)\leq  \frac{1}{2}\ln \left(1+ d_H(A,B)\max\{
\|P(A,R)^{-1}\|, \|P(B,R)^{-1}\|\}\right).
$$
\end{corollary}
Combining Corollary \ref{INE} with Corollary \ref{INE2},  one can
obtain inequalities involving the hyperbolic metric $\delta$ and the
metric induced by the operator norm on $[B_0(\cH)^n]_1^-$. In
particular, if $A,B\in [B(\cH)^n]_1$, then we have
\begin{equation*}
\frac{1}{2} \ln\left( 1+\frac{\|A-B\|}{
\max\{\|P(A,R)\|,\|P(B,R)\|\}}\right) \leq
 \delta(A,B)
\end{equation*}
and
 \begin{equation*} \delta(A,B)\leq \frac{1}{2}\ln \left(1+ \frac{2
\|A-B\|}{(1-\|A\|)(1-\|B\|)}\max\{ \|P(A,R)^{-1}\|,
\|P(B,R)^{-1}\|\}\right).
\end{equation*}

\bigskip

\section{Schwarz-Pick  lemma    with respect to the  hyperbolic metric on $[B(\cH)^n]_1$}

A very important property of the Poincar\' e-Bergman distance
$\beta_m:\BB_m\times \BB_m\to \RR^+$ is that
$$\beta_m(f(z), f(w))\leq \beta_n(z,w),\quad z,w\in \BB_n,
$$
 for any
holomorphic function $f:\BB_n\to \BB_m$. In this section we extend this result and obtain
 a
Schwarz-Pick lemma for free holomorphic functions on $[B(\cH)^n]_1$
with operator-valued coefficients, with respect to the hyperbolic
metric on the noncommutative ball $[B(\cH)^n]_1$.

\begin{lemma} \label{ReRe}
Let $A:=(A_1,\ldots, A_n)$ and $B:=(B_1,\ldots,
B_n)$ be pure  row contractions.  Then
$A\overset{H}{{\underset{c}\prec}}\, B$ if and only if
$$
\text{\rm Re}\,f(A_1,\ldots, A_n)\leq c^2 \text{\rm
Re}\,f(B_1,\ldots, B_n)$$
 for any $f\in F_n^\infty \bar \otimes B(\cE)$ with $
\text{\rm Re}\,f\geq 0$.

\end{lemma}
\begin{proof} Assume that $A,B$ are pure row contractions with
 $A\overset{H}{{\underset{c}\prec}}\, B$
 and let $f\in F_n^\infty \bar \otimes B(\cE)$  be such that  $
\text{\rm Re}\,f\geq 0$. Then  $f$ has  a unique  representation of
the form
$$
f(S_1,\ldots, S_n)=\sum_{k=0}^\infty \sum_{|\alpha|=k}
S_\alpha\otimes A_{(\alpha)},\quad A_{(\alpha)}\in B(\cE).
$$
Due to the results from \cite{Po-holomorphic}, for each $r\in
[0,1)$, $f_r(S_1,\ldots, S_n):=f(rS_1,\ldots, rS_n)$ is in
$\cA_n\otimes B(\cE)$. Moreover, applying the noncommutative Poisson
transform $P_{[rS_1,\ldots, rS_n]}\otimes \text{\rm id}$ to the
inequality $ \text{\rm Re}\,f(S_1,\ldots, S_n)\geq 0$, we deduce
that  $ \text{\rm Re}\,f_r(S_1,\ldots, S_n)\geq 0$. Since
$A\overset{H}{{\underset{c}\prec}}\, B$, Theorem \ref{equivalent}
shows that there exists $c\geq 1$ such that
\begin{equation}
\label{Refr} \text{\rm Re}\,f_r(A_1,\ldots, A_n)\leq c^2 \text{\rm
Re}\,f_r(B_1,\ldots, B_n).
\end{equation}
Due to the $F_n^\infty$-functional calculus for pure row
contractions (see \cite{Po-funct}),
 $$
 f(A_1,\ldots, A_n):=\text{\rm
SOT}-\lim_{r\to 1}f_r(A_1,\ldots, A_n)\quad \text{and}\quad f(B_1,\ldots, B_n):=\text{\rm
SOT}-\lim_{r\to 1}f_r(B_1,\ldots, B_n)
$$
 exist. Consequently, taking the limit, as $r\to 1$,
in inequality \eqref{Refr}, we get
$$
\text{\rm Re}\,f(A_1,\ldots, A_n)\leq c^2 \text{\rm
Re}\,f(B_1,\ldots, B_n).
$$
Since the converse is obvious, the proof is complete.
\end{proof}

Now we prove a Schwarz-Pick lemma for free holomorphic functions on $[B(\cH)^n]_1$
with operator-valued coefficients, with respect to the hyperbolic
metric.

\begin{theorem} \label{S-P} Let $F_j:[B(\cH)^n]_1\to B(\cH) \otimes_{min}
B(\cE)$, $j=1,\ldots, m$, be free holomorphic functions with
coefficients in $B(\cE)$, and assume that $F:=(F_1,\ldots, F_m)$ is
a contractive free holomorphic function. If $X,Y\in [B(\cH)^n]_1$,
then
 $F(X)\overset{H}{\sim}\, F(Y)$ and
$$
\delta(F(X), F(Y))\leq \delta(X,Y),
$$
where $\delta$ is the hyperbolic  metric  defined on the Harnack
parts of the noncommutative ball $[B(\cH)^n]_1^-$.
\end{theorem}

\begin{proof}
Assume that each $F_j$ has  a representation of the form
$$
F_j(X_1,\ldots, X_n)=\sum_{k=0}^\infty \sum_{|\alpha|=k} X_\alpha
\otimes A_{(\alpha,j)},\quad A_{(\alpha,j)}\in B(\cE).
$$
Since $F_j$ is a bounded free holomorphic function, due to
\cite{Po-holomorphic} (see also \cite{Po-pluriharmonic}), there
exists $f_j(S_1,\ldots, S_n)\in F_n^\infty \bar \otimes B(\cE)$ such
that
$$
F_j(X_1,\ldots, X_n)=(P_X\otimes \text{\rm id})[f_j(S_1,\ldots,
S_n)],\quad X:=(X_1,\ldots, X_n)\in [B(\cE)^n]_1.
$$
Moreover,
$$
f_j(S_1,\ldots, S_n)=\text{\rm SOT}-\lim_{r\to 1} F_j(rS_1,\ldots,
rS_n)
$$
and $ F_j(rS_1,\ldots, rS_n)\in \cA_n\otimes B(\cE)$. Now, let
$L_1,\ldots, L_m$ be the left creation operators on the full Fock
space  $F^2(H_m)$ with $m$ generators, and consider
$$
p(L_1,\ldots, L_m)=\sum_{|\alpha|\leq q} L_\alpha\otimes
M_{(\alpha)},\quad M_{(\alpha)}\in B(\CC^k)
$$ to be an arbitrary polynomial with  $\text{\rm Re}\,p(L_1,\ldots,
L_m)\geq 0$. Since $F=(F_1,\ldots, F_m)$ is a contractive free
holomorphic function, we deduce that (see \cite{Po-holomorphic})
$$
\|[f_1(S_1,\ldots, S_n),\ldots, f_m(S_1,\ldots,
S_n)]\|=\|F\|_\infty\leq 1.
$$
Applying  now  the noncommutative  Poisson transform
$P_{(f_1(S_1,\ldots, S_n),\ldots, f_m(S_1,\ldots, S_n))}\otimes
\text{\rm id}$ to the inequality  $\text{\rm Re}\,p(L_1,\ldots,
L_m)\geq 0$, we obtain
\begin{equation}
\label{Re<Re}
 \text{\rm Re}\,p(f_1(S_1,\ldots, S_n),\ldots,
f_m(S_1,\ldots, S_n))\geq 0.
\end{equation}
Note that $p(f(S)):=p(f_1(S_1,\ldots, S_n),\ldots, f_m(S_1,\ldots,
S_n))$ in is $F_n^\infty \bar \otimes B(\cE)\otimes_{min}B(\CC^m)$.

Let  $X:=(X_1,\ldots, X_n)$ and $Y:=(Y_1,\ldots, Y_n)$ be in the
open ball $[B(\cH)^n]_1$. Due to  Theorem \ref{foias2}, we have
$X\overset{H}{\sim}\, Y$. Assume that
$X\overset{H}{{\underset{c}\sim}}\, Y$ for some $c\geq 1$. Due to
\eqref{Re<Re}   and Lemma \ref{ReRe}, we deduce that

$$
\frac{1}{c^2} \text{\rm Re}\,p(F(Y))\leq \text{\rm Re}\,p(F(X))\leq
c^2 \text{\rm Re}\,p(F(Y)).
$$
Consequently, $F(X)\overset{H}{{\underset{c}\sim}}\, F(Y)$ and
$\omega(F(X),F(Y))\leq c$, where $\omega$ is defined by relation \eqref{om}.
 This implies that $\omega(F(X),F(Y))\leq
\omega(X,Y)$, which completes the proof.
\end{proof}

We remark that the hyperbolic metric $\delta$ coincides with the
Carath\' eodory type metric  defined by
$$
C_{\bf ball} (X,Y):=\sup_{F}\delta(F(X),F(Y)),\quad X,Y\in
[B(\cH)^n]_1,
$$
where the supremum  is taken over all free holomorphic functions
$F:[B(\cH)^n]_1\to [B(\cH)^n]_1$. Indeed, due to Theorem \ref{S-P},
we have $C_{\bf ball} (X,Y)\leq \delta(X,Y)$. Taking $F={\rm id}$,
we also deduce that $C_{\bf ball} (X,Y)\geq \delta(X,Y)$, which
proves our assertion.

\begin{corollary} \label{Sch-part}Let $F:=(F_1,\ldots, F_m)$ be
a contractive free holomorphic function  with coefficients in
$B(\cE)$.
 If $z,w\in \BB_n$, then $F(z)\overset{H}{\sim}\, F(w)$ and
$$
\delta(F(z), F(w))\leq \delta(z,w).
$$
\end{corollary}

We remark that if $m=n=1$ in Corollary \ref{Sch-part}, we obtain a
very simple proof of Suciu's result \cite{Su}. Note  also that, in
particular (if $\cE=\CC$), for any free holomorphic function
$F:[B(\cH)^n]_1\to [B(\cH)^m]_1$, we  have
$$
\delta(F(X), F(Y))\leq \delta(X,Y),
$$
which extends the result  mentioned at the beginning  of this
section.

\begin{corollary} If  $f\in H^\infty(\DD)$ is a contractive analytic
function on the open unit disc and  $A,B\in B(\cH)$ are strict
contractions, then $f(A)\overset{H}{\sim}\, f(B)$ and
$$
\delta(f(A), f(B))\leq \delta(A,B).
$$
\end{corollary}

A few remarks are necessary.  Our hyperbolic metric $\delta$ is
different from the Kobayashi distance $\delta_K$ on unit ball
$[B(\cH)^n]_1$. Indeed, when $n=1$, one can show that $\delta$
coincides with the Harnack distance introduced by Suciu. In this
case, according to \cite{Su3} (and due to a result from \cite{Up}),
we have
$$
\delta(0,A)<\delta_K(0,A)=\frac{1}{2}\ln \frac{1+\|A\|}{1-\|A\|}
$$
for certain strict contractions $A\in B(\cH)$  with $\dim \cH\geq
2$. This also shows that $\delta$ is different from the metric  for
the ball  $[B(\cH)^n]_1$, as defined  in \cite{Ha1}.

We  define now a  Kobayashi type pseudo-distance on domains
$M\subset B(\cH)^m$, $m\in \NN$,  with respect to the hyperbolic
metric $\delta$ of the ball $[B(\cH)^n]_1$, as follows. Given two
points $X,Y\in M$, we consider a {\it chain of free holomorphic
balls} from $X$ to $Y$. That is, a chain of elements
$X=X_0,X_1,\ldots, X_k=Y$ in $M$, pairs of elements $A_1,
B_1,\ldots, A_k,B_k$ in $[B(\cH)^n]_1$, and free holomorphic
functions $F_1,\ldots, F_k$ on $[B(\cH)^n]_1$ with values  in $M$
such that
$$
F_j(A_j)=X_{j-1}\ \text{ and } \ F_j(B_j)=X_j\ \text{ for } \
j=1,\ldots, k.
$$
Denote this  chain by $\gamma$ and define its length by
$$
\ell(\gamma):=\delta(A_1,B_1)+\cdots + \delta (A_k, B_k),
$$
where $\delta$ is the hyperbolic metric  on $[B(\cH)^n]_1$. We
define
$$
\delta_{\bf ball}^M(X,Y):=\inf \ell(\gamma),
$$
where the infimum is taken over all chains $\gamma$ of free
holomorphic balls from $X$ to $Y$. If there is no such chain, we set
$\delta_{\bf ball}^M(X,Y)=\infty$. In general,  $\delta_{\bf
ball}^M$ is not  a true  distance  on $M$.  However, it becomes a
true distance  in some special cases.

It is well-known that the Kobayashi distance on the open unit disc
$\DD$ coincides with the Poincar\' e  metric.  A similar result
holds in our noncommutative setting.

\begin{proposition}\label{Koba} If $M=[B(\cH)^n]_1$, then
$\delta_{\bf ball}^M$ is a true distance and \ $ \delta_{\bf
ball}^M=\delta. $
\end{proposition}
\begin{proof}
If $\gamma$ is a chain, as defined above, we use Theorem \ref{S-P}
and the fact that $\delta$ is a metric  to  deduce that
\begin{equation*}
\begin{split}
\delta(X,Y)&\leq  \delta(X_0,X_1)+\delta(X_1,X_2)+\cdots +
\delta(X_{k-1},X_k)\\
&=
 \delta(F_1(A_1),F_1(B_1))+\delta(F_2(A_2),F_2(B_2)+\cdots +
\delta(F_k(A_k),F_k(B_k))\\
&\leq  \delta(A_1,B_1)+\delta(A_2,B_2)+\cdots +
\delta(A_{k},B_k)=\ell(\gamma)
\end{split}
\end{equation*}
Taking the infimum over   all chains $\gamma$ of free holomorphic
balls from $X$ to $Y$, we deduce that $\delta(X,Y)\leq \delta_{\bf
ball}^M(X,Y)$. Taking $F$  the identity on $[B(\cH)^n]_1$, we obtain
$\delta_{\bf ball}^M(X,Y)\leq \delta(X,Y)$.
\end{proof}

It would be interesting  to find, as in the classical case,  classes
of  noncommutative domains $M$ in $B(\cH)^m$ so  that $\delta_{\bf
ball}^M$ is a true distance.

\bigskip

      %\Refs
      %\widestnumber\key{BFPQR}
      %\def\n{\key}
       %


\begin{thebibliography}{99}


%\bibitem{Ar} {\sc W.B.~Arveson},
% Subalgebras of $C^*$-algebras,
 %{\it Acta Math.}
 %{\bf 123} (1969),  141--224.



\bibitem{AST} {\sc T.~Ando, I.~Suciu, D.~ Timotin},
 Characterization of some Harnack parts of contractions,
{\it J. Operator Theory} {\bf  2} (1979), no. 2, 233--245.




\bibitem{Be} {\sc S.~ Bergman},
 {\it The kernel function and conformal mapping},
 Mathematical Surveys, No. V. American Mathematical Society,
Providence, R.I., 1970. x+257 pp.



\bibitem{Co} {\sc J.B.~Conway},
{\em Functions of one complex variable. I.} Second Edition. Graduate
Texts in Mathematics  {\bf 159}. { Springer-Verlag, New York}, 1995.


\bibitem{Cu} {\sc J.~Cuntz},
 Simple $C^*$--algebras generated by isometries,
 {\it  Commun. Math. Phys.}
 {\bf 57} (1977), 173--185.


\bibitem{DP2} {\sc K.~R.~Davidson and D.~Pitts},
%
 The algebraic structure of  non-commutative
  analytic Toeplitz algebras,
{\it  Math. Ann.}
   {\bf 311} (1998),  275--303.



            \bibitem{DP1} {\sc K.R.~Davidson and D. ~Pitts},
      {Invariant subspaces and hyper-reflexivity for free semigroup
      algebras,}
      {\it Proc. London Math. Soc.} {\bf 78} (1999), 401--430.



\bibitem{ER} {\sc E.G.~Effros and Z.J.~Ruan},
  {\em Operator spaces},
 London Mathematical Society Monographs. New Series, {\bf 23}.
 The Clarendon Press, Oxford University Press, New York, 2000.



\bibitem{Fo} {\sc C.~ Foia\c s},  On Harnack parts of contractions.
 {\it Rev. Roumaine Math. Pures Appl.} {\bf  19} (1974), 315--318.

\bibitem{Ga} {\sc J.~Garnett},
 {\it  Bounded analytic functions},
 Academic Press, New York, 1981.




\bibitem{G} {\sc A.M.~Gleason}, {\it Function algebras}, Seminars on
Analytic Functions, Vol. 2, Institute for advanced Study, Princeton,
N.J., 1957.


\bibitem{Ha1} {\sc L.A.~Harris},
 Bounded symmetric homogeneous domains in infinite dimensional spaces,
 {\it Proceedings on Infinite Dimensional Holomorphy (Internat. Conf.,
 Univ. Kentucky, Lexington, Ky.}, 1973),
 pp. 13--40. {\it Lecture Notes in Math.}, Vol. 364, Springer, Berlin, 1974.


\bibitem{Ha2} {\sc L.A.~Harris},
Schwarz-Pick systems of pseudometrics for domains in normed linear
spaces, {\it Advances in holomorphy (Proc. Sem. Univ. Fed. Rio de
Janeiro, Rio de Janeiro}, 1977), pp. 345--406, North-Holland Math.
Stud., {\bf 34}, North-Holland, Amsterdam-New York, 1979.



\bibitem{Ha3} {\sc L.A.~Harris},
Analytic invariants and the Schwarz-Pick inequality., {\it Israel J.
Math.} {\bf 34} (1979), no. 3, 177--197.



\bibitem{H} {\sc K.~Hoffman},  {\em Banach Spaces of Analytic Functions},
Englewood Cliffs: Prentice-Hall, 1962.

\bibitem{Ko1} {\sc S.~Kobayashi}, {\it Hyperbolic complex spaces}, Grundlehren der Mathematischen Wissenschaften [Fundamental Principles of Mathematical Sciences], 318.
 Springer-Verlag, Berlin, 1998. xiv+471 pp.

\bibitem{Ko2} {\sc S.~Kobayashi}, {\it  Hyperbolic manifolds and holomorphic mappings},
 World Scientific Publishing Co. Pte. Ltd., Hackensack, NJ, 2005. xii+148 pp.



\bibitem{Kr} {\sc S.G.~Krantz}, {\it  Geometric function theory},
 Explorations in complex analysis.
 Cornerstones. Birkhäuser Boston, Inc., Boston, MA, 2006. xiv+314 pp.



\bibitem{Pa-book} {\sc V.I.~Paulsen},
 {\it Completely Bounded Maps and Dilations},
Pitman Research Notes in Mathematics, Vol.146, New York, 1986.




\bibitem{Ph} {\sc R.S.~Phillips}, On symplectic mappings of
contraction operators, {\it Studia Math.} {\bf 31} (1968), 15--27.

\bibitem{Pi} {\sc G.~Pisier}, \emph{ Similarity Problems and Completely Bounded Maps},
Springer Lect. Notes Math., Vol.1618, Springer-Verlag, New York,
1995.





\bibitem{Po-multi} {\sc G.~Popescu},
 Multi-analytic operators and some  factorization theorems,
 {\it Indiana Univ. Math.~J.}
 {\bf 38} (1989),   693--710.

\bibitem{Po-charact} {\sc G.~Popescu}, Characteristic functions for infinite
sequences of noncommuting operators, {\it J. Operator Theory}
{\bf 22} (1989), 51--71.


      \bibitem{Po-von} {\sc G.~Popescu},
{Von Neumann inequality for $(B(H)^n)_1$,}
      {\it Math.  Scand.} {\bf 68} (1991), 292--304.
      %%\bigskip




      \bibitem{Po-funct} {\sc G.~Popescu},
      {Functional calculus for noncommuting operators,}
       {\it Michigan Math. J.} {\bf 42} (1995), 345--356.



      \bibitem{Po-analytic} {\sc G.~Popescu},
      {Multi-analytic operators on Fock spaces,}
      {\it Math. Ann.} {\bf 303} (1995), 31--46.




      \bibitem{Po-disc}  {\sc G.~Popescu},
      {Noncommutative disc algebras and their representations,}
       {\it Proc. Amer. Math. Soc.} {\bf 124} (1996),  2137--2148.





      \bibitem{Po-poisson} {\sc G.~Popescu},
     {Poisson transforms on some $C^*$-algebras generated by isometries,}
       {\it J. Funct. Anal.} {\bf 161} (1999),  27--61.



\bibitem{Po-curvature} {\sc  G.~Popescu},
  Curvature invariant for Hilbert modules over free semigroup algebras,
   {\it Adv. Math.}
 {\bf 158} (2001), 264--309.




\bibitem{Po-varieties} {\sc  G.~Popescu},
Operator theory on noncommutative varieties, {\it Indiana Univ.
Math.~J.} {\bf 55}, No.2, (2006), 389--442.




      \bibitem{Po-holomorphic} {\sc G.~Popescu},
      {Free holomorphic functions on the unit ball of $B(\cH)^n$},
      {\it J. Funct. Anal.}  {\bf 241} (2006), 268--333.


\bibitem{Po-free-hol-interp} {\sc G.~Popescu},
{ Free holomorphic functions and interpolation},
 {\it Math. Ann.} {\bf 342} (2008), 1-30.




\bibitem{Po-majorants} {\sc G.~Popescu},
{Free  pluriharmonic majorants and commutant lifting}, {\it J.
Funct. Anal.}  {\bf 255} (2008), no. 4, 891--939.



\bibitem{Po-pluriharmonic} {\sc G.~Popescu},
{Noncommutative transforms and free pluriharmonic functions}, {\it
Adv. Math.}, {\bf 220} (2009), 831-893


      \bibitem{Po-unitary} {\sc G.~Popescu},
      {Unitary invariants in multivariable operator theory},
      {\it Mem. Amer. Math. Soc.}, to appear.


\bibitem{Po-automorphism} {\sc G.~Popescu},
{ Free holomorphic automorphisms of the unit ball of $B(\cH)^n$},
{\it J. Reine Angew. Math.}, to appear.


\bibitem{Po-domains} {\sc G.~Popescu},
Operator theory on noncommutative domains, {\it Mem. Amer. Math.
Soc.}, to appear.

\bibitem{Ru} {\sc W.~Rudin},
{\em Function theory in the unit ball of \,$\CC^n$}, {
Springer-verlag, New-York/Berlin}, 1980.






\bibitem{Su} {\sc I.~Suciu}, Analytic relations between functional models for
contractions, {\it Acta Sci. Math. (Szeged)} {\bf  34} (1973),
359--365.


\bibitem{Su2} {\sc I.~Suciu}, Analytic formulas  for the hyperbolic
distance between two contractions, {\it  Ann. Polon. Math.} {\bf 66}
(1997), 239--252.


 \bibitem{Su3} {\sc I.~Suciu},
 The Kobayashi distance between two contractions.
 {\it Operator extensions, interpolation of functions and related topics }(Timisoara, 1992),
  189--200,
   Oper. Theory Adv. Appl.  {\bf 61}, Birkhäuser, Basel, 1993.


\bibitem{SzF-book} {\sc B.~Sz.-Nagy and C.~Foia\c{s}}, {\em Harmonic
Analysis of Operators on Hilbert Space}, North Holland, New York
1970.


      \bibitem{vN}  {\sc J.~von Neumann},
      {Eine Spectraltheorie f\"ur allgemeine Operatoren eines unit\"aren
      Raumes,}
      {\it Math. Nachr.} {\bf 4} (1951), 258--281.
      %%\bigskip



\bibitem{Up} {\sc U.~Upmeier}, {\em Symmetric Banach manifolds and
Jordan $C^*$-algebras}, North-Holland Mathematics Studies {\bf 104},
North Holland, 1985.

\bibitem{Y} {\sc N.J.~Young}, Orbits of the unit sphere of
$\cL(\cH,\cK)$ under symplectic transformations, {\it J. Operator
Theory} {\bf 11} (1984), 171--191.


\bibitem{Zhu} {\sc K.~Zhu}, {\it Spaces of holomorphic functions in the unit
ball},  Graduate Texts in Mathematics, 226. Springer-Verlag, New
York, 2005. x+271 pp.

       \end{thebibliography}
 \end{document}